\documentclass[english,
	twoside,
	usegeometry=true,
	abstract=on,]{scrartcl}

\pdfoutput = 1

\usepackage[
	includehead,
	top    = 2cm,
	bottom = 2.95cm,
	left   = 3.25cm,
	right  = 3.25cm,
]{geometry}

\usepackage[english]{babel}
\usepackage[utf8]{inputenc}
\usepackage[autostyle=true]{csquotes}

\usepackage[T1]{fontenc}
\usepackage{ae}
\usepackage{aecompl}
\usepackage{lmodern}
\usepackage{textcomp}
\usepackage{dsfont}
\usepackage[final]{microtype}
\usepackage[onehalfspacing]{setspace}
\AfterTOCHead{\linespread{1.2}}
\usepackage{xspace}

\newcommand{\SageMath}{\texttt{SageMath}\xspace}
\newcommand{\Sage}{\texttt{Sage}\xspace}
\newcommand{\SageManifolds}{\texttt{SageManifolds}\xspace}
\newcommand{\Python}{\texttt{Python}\xspace}

\newcommand{\ticket}[1]{\href{https://trac.sagemath.org/ticket/#1}{\#{#1}}}

\usepackage{amsthm}
\usepackage{amsmath}
\usepackage{mathtools}
\numberwithin{equation}{section}
\usepackage{amssymb}
\usepackage{amsfonts}
\usepackage{xfrac}
\usepackage{faktor}

\newtheorem{mythm}[equation]{Theorem}

\newtheorem{mylem}[equation]{Lemma}

\theoremstyle{definition}

\newtheorem{mydef}[equation]{Definition}

\newcommand{\scal}[2]{\left< #1 , #2 \right>}

\newcommand{\abs}[1]{\left| #1 \right|}

\newcommand{\restrict}[2]{\left. #1 \right|_{#2}}

\newcommand{\im}{\mathrm{im}}

\newcommand{\Bold}[1]{\mathbf{#1}}

\newcommand{\proj}{\mathrm{pr}}

\newcommand{\diffd}{\mathrm{d}}

\newcommand{\expe}{\mathrm{e}}

\newcommand{\tr}{\mathrm{tr}}
\newcommand{\Pf}{\mathrm{Pf}}
\newcommand{\ind}{\mathrm{ind}}
\newcommand{\id}{\mathrm{id}}

\newcommand{\NN}{\mathbb{N}}
\newcommand{\ZZ}{\mathbb{Z}}

\newcommand{\RR}{\mathbb{R}}
\newcommand{\CC}{\mathbb{C}}
\newcommand{\KK}{\mathbb{K}}
\newcommand{\HH}{\mathbb{H}}

\newcommand{\CP}{\CC\mathbb{P}}

\newcommand{\Sphere}{\mathbb{S}}

\newenvironment{smallpmatrix}
{\left(\begin{smallmatrix}}
	{\end{smallmatrix}\right)}

\usepackage{graphicx}
\usepackage{tabularx}
\usepackage{booktabs}
\usepackage{colortbl}
\usepackage[normal,
	width=.8\textwidth]{caption}
\addtokomafont{captionlabel}{\bfseries}		%

\usepackage{doi}
\usepackage{cleveref}
\usepackage[backend=biber,
			giveninits=true,
		    natbib=true,
			style=numeric-comp,
]{biblatex}

\usepackage[dvipsnames]{xcolor}
\definecolor{kwcolor}{HTML}{3a821d}
\definecolor{strcolor}{HTML}{a03626}
\definecolor{ccolor}{HTML}{477998}
\newcommand{\myshade}{75}

\usepackage{pdfpages}

\newcommand{\defstyle}[1]{\textit{\textbf{#1}}}

\usepackage[markcase=noupper,%
			headsepline=true,
]{scrlayer-scrpage}

\automark[section]{subsection}
\automark*[subsection]{}
\clearpairofpagestyles
\addtokomafont{pagehead}{\sffamily\normalfont}
\lehead{\textit{\rightmark}}
\rehead{\pagemark}
\lohead{\leftmark}
\rohead{\pagemark}

\newpairofpagestyles[scrheadings]{front}{\KOMAoptions{headsepline=false}\clearpairofpagestyles\cfoot{\pagemark}}

\KOMAoptions{onpsinit={\linespread{1}\selectfont}}

\usepackage{sectsty}

\sectionfont{\centering\large\bfseries\lsstyle\MakeUppercase}
\subsectionfont{\centering}
\usepackage{indentfirst}

\usepackage{tikz}
\usetikzlibrary{arrows,shapes,fit,backgrounds,calc,positioning}
\usepackage{tikz-cd}

\usepackage{lipsum}

\usepackage{enumitem}

\newlist{myenum}{enumerate}{2}
\setlist[myenum,1]{label=\roman*)}
\setlist[myenum,2]{label=\alph*)}

\newlist{Steps}{enumerate}{1}
\setlist[Steps,1]{label=\bfseries Step~\arabic*.,leftmargin=*}

\usepackage[many,breakable,listings]{tcolorbox}
\tcbset{nobeforeafter} 

\definecolor{incolor}{HTML}{234c66}
\definecolor{outcolor}{HTML}{b73e2c}
\definecolor{cellborder}{HTML}{CFCFCF}
\definecolor{cellbackground}{HTML}{F7F7F7}

\newcommand{\wshade}{90}

\newlength{\promptwidth}
\newcommand{\prompt}[4]{%
	\makebox[0pt][r]{\texttt{\color{#2}#1[#3]:#4}}\vspace{-\baselineskip}%
}

\newtcblisting{NBout}{
	breakable,
	enhanced,
	boxrule=.5pt, 
	size=fbox, 
	pad at break*=1mm, 
	left skip = \promptwidth,
	opacityfill=0,
	title=\prompt{}{outcolor!\wshade!white}{\theNBin}{\hspace{5.4pt}},
	fonttitle=\linespread{1}\small,
	attach title to upper,
	after={\par\noindent},
	listing only,
	listing options={
		basicstyle=\linespread{1}\small\ttfamily,
		upquote= true,
		basewidth=.5em,
		aboveskip=0pt,
		belowskip=0pt,
		showstringspaces=false,
		breaklines=true,
	}
}

\newtcolorbox{NBoutM}{
	breakable,
	enhanced,
	boxrule=.5pt,
	left skip = \promptwidth,
	size=fbox, 
	pad at break*=1mm, 
	opacityfill=0,
	title=\prompt{}{outcolor!\wshade!white}{\theNBin}{\hspace{5.4pt}},
	fonttitle=\linespread{1}\small,
}

\newcounter{NBin}
\newtcblisting{NBin}{%
	breakable,
	enhanced,
	size=fbox, 
	pad at break*=1mm,
	opacityfill=1,
	left skip = \promptwidth,
	boxrule=1pt,
	before skip=\bigskipamount,
	after={\par\noindent},
	colback=cellbackground,
	colframe=cellborder,
	phantom=\refstepcounter{NBin}\label{in:\theNBin},
	title=\prompt{}{incolor!\wshade!white}{\theNBin}{\hspace{6pt}},
	fonttitle=\linespread{1}\small,
	attach title to upper,
	listing only,
	listing options={
		language=Python,
		upquote= true,
		tabsize=4,
        keywordstyle=\bfseries\color{kwcolor},
		commentstyle=\color{ccolor},
		stringstyle=\color{strcolor},
		basicstyle=\linespread{1}\small\ttfamily,
		basewidth=.5em,
		aboveskip=0pt,
		belowskip=-2pt,
		showstringspaces=false,
		breaklines=true,
		sensitive=true,
		literate=%
		*{0}{{{\color{kwcolor}0}}}1
		{1}{{{\color{kwcolor}1}}}1
		{2}{{{\color{kwcolor}2}}}1
		{3}{{{\color{kwcolor}3}}}1
		{4}{{{\color{kwcolor}4}}}1
		{5}{{{\color{kwcolor}5}}}1
		{6}{{{\color{kwcolor}6}}}1
		{7}{{{\color{kwcolor}7}}}1
		{8}{{{\color{kwcolor}8}}}1
		{9}{{{\color{kwcolor}9}}}1
		{+}{{{\color{gray}+}}}1
		{-}{{{\color{gray}-}}}1
		{=}{{{\color{gray}=}}}1
		{*}{{{\color{gray}*}}}1
		{/}{{{\color{gray}/}}}1
		{restrictions1}{{restrictions1}}{13}
		{restrictions2}{{restrictions2}}{13}
		{.set}{{.set}}{4}
		{Omega0}{{Omega0}}{6}
		{Omega1}{{Omega1}}{6}
		{Omega2}{{Omega2}}{6}
		{omega1}{{omega1}}{6}
		{omega2}{{omega2}}{6}
		{C0}{{C0}}{2}
		{C0U}{{C0U}}{3}
		{phiT11U}{{phiT11U}}{7}
	},
}

\hypersetup{
	colorlinks = true,
	linktocpage = false,
	linkcolor  = kwcolor!\myshade!black,
	citecolor  = strcolor!\myshade!black,
	urlcolor   = ccolor!\myshade!black,
  	pdfinfo={
		Title={Master's Thesis to Obtain the Degree of Master of Science (M.Sc.) in Mathematics},
		Author={Michael Jung},
		Subject={Characteristic Classes in Computer Algebra},
	}
}

\begin{document}

\pagenumbering{gobble}


# MASTER'S THESIS

## Characteristic Classes in Computer Algebra

**MICHAEL JUNG**
Characteristic Classes in Computer Algebra

Date:                         02/2020
University:                   University of Potsdam
Department:                   Department of Mathematics
Matriculation Number:         761691
First Supervisor:             Prof. Dr. Christian Bär
Second Supervisor:            Dr. Saskia Roos

# CHARACTERISTIC CLASSES IN COMPUTER ALGEBRA

**Zusammenfassung**

In dieser Arbeit entwickeln wir ein Modell, um charakteristische Klassen und ihre Formen unter Verwendung symbolischer Berechnungen in das Computer Algebra System `SageMath` einzubetten. Dabei greifen wir auf die Chern–Weil-Methode zurück, bei der charakteristische Klassen auf Vektorbündeln in der de Rham Kohomologie aus beliebigen Bündelzusammenhängen gewonnen werden können. Auf dem Weg dorthin implementieren Vektorbündel, ihre Schnitte und Zusammenhänge, sowie gemischte Differentialformen. Im Anschluss diskutieren wir unsere Implementierungen mit Hilfe von Anwendungsbeispielen im `Jupyter Notebook` und vergleichen ihre Berechnungszeiten.

**Selbstständigkeitserklärung**

Hiermit erkläre ich, **Michael Jung**, dass ich die vorliegende schriftliche Prüfungsleistung selbstständig und nur mit den angegebenen Quellen und Hilfsmitteln (einschließlich elektronischer Medien und Online-Quellen) angefertigt habe. Alle Stellen der Arbeit, die ich aus diesen Quellen und Hilfsmitteln dem Wortlaut oder dem Sinne nach entnommen habe, sind kenntlich gemacht und im Literaturverzeichnis aufgeführt. Weiterhin versichere ich, dass weder ich noch andere diese Arbeit weder in der vorliegenden noch in einer abgewandelten Form als Leistungsnachweise in einer anderen Veranstaltung bereits verwendet haben oder noch verwenden werden. Mir ist bewusst, dass ein Täuschungsversuch oder ein Ordnungsverstoß vorliegt, wenn sich diese Erklärung als unwahr erweist.

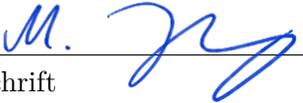
Unterschrift

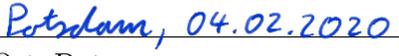
Ort, Datum


\clearpage

\thispagestyle{front}
\renewcommand\abstractname{Abstract}
\begin{abstract}
	We develop a framework to compute characteristic classes and their forms in the computer algebra system \SageMath using symbolic calculus. In order to do this, we make use of the Chern--Weil approach, in which characteristic classes of vector bundles in the de~Rham cohomology are obtained by arbitrary connections. Along the way, we implement the notion of vector bundles, their sections and connections as well as mixed differential forms in \SageMath. We conclude by discussing some application examples exposed in \texttt{Jupyter Notebook} and eventually address the issue of computational cost.
\end{abstract}
\medskip
\subsection*{Preface and Acknowledgments}
Foremost, I would like to thank Christian Bär for offering me the opportunity to write my master’s thesis on this fascinating topic with him. He constantly gave me new inputs and provided me with further literature. I especially thank him for his patience and constant motivation.

The whole development in \Sage is collaboratively organized at its trac server \url{https://trac.sagemath.org/}, where code proposals undergo a rigorous review process. On that account, I would like to thank Travis Scrimshaw and especially Éric Gourgoulhon for their discussion, their attempt to answer my naive questions and for reviewing my code. Thanks to them, the code is now part of the official \Sage project.

On that occasion I am grateful, and at the same time indepted, to all my proofreaders who have found each unnecessary \enquote{of course}, explained the English comma rules, and gave me an adequate feedback. In this context, I particularly would like to mention Lashi Bandara and Katharina Kuckla, who supported me in improving the presentation, as well as Florian Hanisch, Andreas Hermann, Saskia Roos and Hemanth Saratchandran, who also gave me valuable feedback.

Note that the breadth of topics in this thesis does not reflect the effort of work during the development stage. The primary goal of this thesis is a working algorithm for characteristic classes completely embedded in \Sage. Hence, an essential part of this thesis is devoted to that topic. Regardless, the bulk of the work was dedicated to building a firm foundation rather than characteristic classes themselves; meaning vector bundles and mixed differential forms. During this process, new ideas kept coming up making it difficult to find an appropriate finish line. However, I am confident that these ideas will find a place in future versions of \Sage.
\clearpage

\thispagestyle{front}
\tableofcontents
\clearpage

\pagenumbering{arabic}
\newgeometry{
	includehead,
	top    = 2cm,
	bottom = 2.95cm,
	left   = 3.25cm, %
	right  = 3.25cm, %
}

\section{Introduction}
\subsection{SageMath: An Overview}
Nowadays it is nearly impossible to avoid the use of computer algebra systems in various fields of science. They have proved to be an essential tool for the verification of scientific results and the testing of scientific hypotheses. At the same time, their use in computational aspects of mathematics has led to significant theorems within various fields of mathematics. By now, there are many computer algebra systems on the market like \texttt{Magma}, \texttt{Maple}, \texttt{Mathematica} and \texttt{MATLAB}. However, their source codes are typically under a proprietary license and can therefore not be examined, verified or altered. Consequently, mathematicians are unallowed to adapt the code for their own research, which particularly turns out to be a problem when it comes to very specialized areas.

The computer algebra system \SageMath or \Sage (\enquote{\textbf{S}ystem for \textbf{A}lgebra and \textbf{G}eometry \textbf{E}xperimentation}) aims to combine the advantages of all existing algebra systems and offers an adaptable open-source tool for research and teaching in mathematics. It makes use of nearly 100 open-source packages such as \texttt{Pynac}, \texttt{Maxima}, \texttt{GAP}, \texttt{SymPy}, \texttt{matplotlib} and \texttt{Singular}, which are all interfaced via the \Python programming language. Apart from that, \Sage contains a vast library of additional code, which provides functionalities in advanced mathematics like matroids, modular forms or manifolds. We provide a brief overview, but more information can be found on the official web page~\cite{sagemath}.

\paragraph{Some Notes on History.}
Created in 2005 by William Stein and strongly influenced by \texttt{Magma}, \Sage was designed as an open-source alternative to non-free computer algebra systems~\cite{stein:sage}. Since then, it has been adapted, improved and extended by over 271 developers from all over the world -- mostly by mathematicians~\cite{sagemath:developer}. Notably in 2013, Éric Gourgoulhon, Michał Bejger and Marco Mancini developed an additional package providing differentiable manifolds and extensive symbolic tensor algebra~\cite{gour:sagemanifolds}. The project became officially known under the name \texttt{SageManifolds} and has been predominantly used for applications in general relativity~\cite{gour:sagemanifolds:website:publications}. As of 2016, it is fully embedded into \Sage and has since been developed further. At this stage, it supports scalar fields, tensor fields, differential forms, connections and more~\cite{gour:sagemanifolds:website}.

\paragraph{Capabilities and Limitations.}
As \Sage is based on the \Python programming language, it inherits all functionalities from \Python, in particular object-oriented programming. Apart from an extensive symbolic calculus making use of a huge library of algorithms imported from \texttt{Maxima} and \texttt{Singular}, the \Sage project offers a slightly modified syntax that is closer to actual mathematics. In \Python for example, integers are represented by the type \texttt{int} and real numbers by \texttt{float}. This approach is not suitable for mathematics since there are infinitely many algebraic structures which are not feasible to implement via individual classes. Due to this reason, \Sage has been set up with a parent-element-type pattern instead. \defstyle{Parents} represent mathematical sets endowed with certain structures, while \defstyle{elements} reflect their members~\cite{sagemath:reference:primer}. Specifically, each parent belongs to a dynamically generated class, which represents the parent's mathematical category. Within this setting, \Sage performs automatic conversions prior to binary operations according to strict mathematical rules~\cite{sagemath:reference:coercions}.\footnote{In proper \Python syntax, there is typically no way to perform a common operation on objects with different types, even if both are algebraically compatible. This is due to the fact that \Python is a \emph{strongly typed} language.} Such conversions are called \defstyle{coercions} and portray the key concept of \Sage. A full record of \Sage's capabilites and limitations can be found in~\cite{sagemath:reference}.

\subsection{Aims and Scopes}
Characteristic classes play a substantial role in many parts of mathematics and physics. They are related to indices of elliptic operators and yield differential geometric as well as topological invariants. However, computations by hand are often tedious and can easily introduce errors, mostly due to the non-commutativity of the cup product. Their complexity rapidly increases with higher dimensions. Already in dimension four, the physically interesting case, such computations can be extremely difficult and time consuming. For this reason, we aim for a working computer algorithm, so that we can test our theories and verify results in a time efficient manner. One attempt has been followed by Oleksandr Iena in 2015, performing symbolic computations on the Chern roots in \texttt{Singular} by using multiplicative sequences~\cite{iena:chern}. Subsequently in 2019, Zsolt Szilágyi provided an algorithm in \texttt{Singular} for the calculation of Chern classes associated to the tensor product of vector bundles~\cite{szilgyi:chern}. These algorithms are relatively fast and provide an efficient approach for purely algebraic purposes. Yet, in order to evaluate integrals on characteristic classes, we prefer a more geometric approach yielding results in concrete coordinates. As motivated in the preceding section, the most promising candidate to realize such a project is given by \Sage; it is open-source and comes with a comprehensive symbolic tensor calculus delivered by the \SageManifolds extension.

\paragraph{Our Work.}
As part of this thesis, our goal is to implement characteristic classes into \Sage by using Chern--Weil theory. It provides a construction for characteristic classes in the de Rham cohomology of the underlying base space by inserting curvature forms into invariant polynomials. However, before we transfer this algorithm into \Sage, some preparations have to be made. On the one hand, we require vector bundles including local frames and bundle connections. On the other hand, since invariant polynomials must be applied to curvature forms, a new algebraic structure representing mixed differential forms has to be established. Meanwhile, some parts of the existing code need to be adjusted or even corrected if necessary. After laying those foundations, characteristic classes fit into the setup quite naturally. The whole development is accompanied by a critical review process maintained by other \Sage developers. A list of all modifications due to this project is given in Table~\ref{tab:intro:tickets}. The comprehensive git difference statistics can be found on \url{https://trac.sagemath.org/} and is also provided on the CD attached to this thesis.

\begin{table}[t]
	\centering
	\footnotesize
	\begin{tabular}{lccr}
		\toprule
		\textbf{Implementation Task} & \textbf{Type} & \textbf{Ticket} & \textbf{Merged in} \\
		\midrule
		Output dicts and lists copied & \textit{Bugfix} & \ticket{28563} & \texttt{Sage 9.0.beta1} \\
		Scalar fields: wedge product fix & \textit{Bugfix} & \ticket{28579} & \texttt{Sage 9.0.beta1} \\
		Better treatment of zero element & \textit{Performance} & \ticket{28562} & \texttt{Sage 9.0.beta3} \\
		Scalar fields: restrictions & \textit{Enhancement} & \ticket{28554} & \texttt{Sage 9.0.beta4} \\
		Tensor fields: \texttt{set\_restriction} fixed & \textit{Bugfix} & \ticket{28628} & \texttt{Sage 9.0.beta4} \\
		Tensor fields: consistent naming & \textit{Enhancement} & \ticket{28564} & \texttt{Sage 9.0.beta7} \\
		Automorphism fields: notation of inverse & \textit{Bugfix} & \ticket{28973} & \texttt{Sage 9.1.beta1} \\
		\arrayrulecolor{darkgray}\hline
		Vector bundles and sections & \textit{Enhancement} & \ticket{28159} & \texttt{Sage 9.0.beta3} \\
		Vector bundles: section module name fix & \textit{Bugfix} & \ticket{28690} & \texttt{Sage 9.0.beta5} \\
		\arrayrulecolor{darkgray}\hline
		Mixed differential forms & \textit{Enhancement} & \ticket{27584} & \texttt{Sage 8.8.beta3} \\
		Mixed forms: code improvements & \textit{Enhancement} & \ticket{28578} & \texttt{Sage 9.0.beta5} \\
		Mixed forms: improved coercions & \textit{Enhancement} & \ticket{28916} & \texttt{Sage 9.1.beta1} \\
		Mixed forms: better treatment of zero/one & \textit{Performance} & \ticket{28921} & \texttt{Sage 9.1.beta1} \\
		\arrayrulecolor{darkgray}\hline
		Characteristic classes & \textit{Enhancement} & \ticket{27784} & \texttt{Sage 9.0.beta8} \\
		\bottomrule
	\end{tabular}
	\caption{All implementations developed as part of this thesis. A full description and git difference statistics can be found either on \Sage's trac server \url{https://trac.sagemath.org/} or on the CD attached to this thesis.}\label{tab:intro:tickets}
\end{table}

\paragraph{Outline of this Thesis.}
In this thesis, we present our implementations delivered to \Sage. Furthermore, the thesis is partly intended as a compendium. In Chapter~\ref{ch:vec}, we start by recalling basic facts about vector bundles. We briefly explain their realization in \Sage and provide two illustrative examples. In Chapter~\ref{ch:mixed}, we introduce mixed differential forms and investigate their algebraic structure, highlighting important properties in anticipation of characteristic classes. We sketch the idea behind their algebraic realization in \Sage and give elementary examples exemplifying their usage. The subsequent Chapter~\ref{ch:char_class} addresses to characteristic classes. There, we provide a more detailed but still compressed introduction into Chern--Weil theory. Afterwards, we thoroughly discuss our algorithm in \Sage and eventually provide three simple examples. One last example is devoted to a more sophisticated computation of an $\hat{A}$-form on a Lorentzian foliation of Berger 3-spheres. This chapter is closed with a summary of our work and future prospects.

During this thesis, we assume basic knowledge about differential geometry. Introductory as well as advanced literature on that topic can be found in~\cite{baer:diffgeo, greub:con, hirsch:differential_topology, lee:smooth_manifolds, walschap:metric_structures}. To fully understand our implementation details, we require some understanding of object-oriented programming. A short course with respect to \texttt{Python 3} is dedicated to~\cite{realpython:object}. Moreover, we recommend to have a vague idea behind the \Python module structure of \Sage, especially with respect to \texttt{sage.manifolds} and \texttt{sage.tensor.modules}. The full module index is provided in~\cite{sagemath:reference:module_index}. For our \Sage examples, we expect rudimentary knowledge about \Python's syntax. Free \texttt{Python 3} introduction courses can be found in~\cite{python_everybody, python_in_10}. As aforementioned, \Sage's syntax slightly differs to that of \Python's. We therefore recommend~\cite{zimmermann:sagemath, stein:sagetut} and the reference manual~\cite{sagemath:reference}. An extensive documentation as well as a wide range of examples on the usage of manifolds in particular can be found in~\cite{gour:tensor_lecture, gour:examples, gour:sagemanifolds:website:doc, sagemath:reference:manif}.

\subsection{Conventions and Setup}

\paragraph{The Mathematical Setup.}
An $n$-dimensional manifold is a topological space $M$ obeying the following axioms:
\begin{myenum}
	\item $M$ is \defstyle{locally homeomorphic} to $\RR^n$, i.\,e.\ each point $p \in M$ has a neighborhood being homeomorphic to an open subset of $\RR^n$.
	\item $M$ is a \defstyle{Hausdorff space}, i.\,e.\ any two distinct points $p,q \in M$ have disjoint open neighborhoods.
	\item $M$ is \defstyle{second-countable}, i.\,e.\ $M$ admits a countable family of open subsets $(\mathcal{B}_i)_{i \in I}$ such that every open set on $M$ can be written as a union of members of this family.
\end{myenum}
Beware that in some literature the latter condition is replaced with paracompactness. Even though manifolds can be considered over more general fields, we want to restrict ourselves to the most common case $\RR$. If a topological manifold $M$ admits a differentiable structure, that is, $M$ is covered by charts such that all transition maps are differentiable up to a certain degree, we say $M$ is a \defstyle{differentiable manifold}. As usual, manifolds having a smooth structure are called \defstyle{smooth manifolds}. Note that any manifold endowed with a differentiable structure automatically admits a compatible \emph{smooth} structure~\cite[Thm.~2.9]{hirsch:differential_topology}. We henceforth distinguish only between topological and smooth structures.

\paragraph{The SageMath Setup.}
At the end of each chapter, we illustrate our implementations with examples in \Sage. For this, we use \texttt{Jupyter Notebook}, cf.~\cite{jupyter}, running on the \Sage kernel. The associated \texttt{ipynb}-file with reduced explanations is included on the CD. Within the notebook, inputs are framed with a gray box and designated with blue numbers in brackets:
\begin{NBin}
print('Hello World!')
\end{NBin}
The corresponding output is indicated by the same number in red:\medskip
\begin{NBout}
Hello World!
\end{NBout}
Of course, we are using the most recent version of \Sage so that all our written code is accessible:
\begin{NBin}
version()
\end{NBin}
\begin{NBout}
'SageMath version 9.1.beta1, Release Date: 2020-01-21'
\end{NBout}
Manifolds can be declared by stating the manifold's dimension and name with the following command:
\begin{NBin}
Manifold(2, 'M')
\end{NBin}
\begin{NBout}
2-dimensional differentiable manifold M
\end{NBout}
For mathematical purposes, we want to have a \LaTeX-typeset output by default. This can be achieved with the following line:
\begin{NBin}
%display latex
\end{NBin}

\section{Vector Bundles}\label{ch:vec}

\subsection{Mathematical Preliminaries}
This section is dedicated to the key concept of this thesis: vector bundles. We review basic definitions and briefly outline structures such as {sections}, {bundle metrics} and {bundle connections}. For a comprehensive discussion about vector bundles, we refer to~\cite[Ch.~10]{lee:smooth_manifolds}, \cite[§2]{milnor:char_classes},~\cite[Ch.~5]{baer:diffgeo}. We start with the definition of a vector bundle.
\begin{mydef}
	Let $M$ be a topological manifold and $\KK=\RR,\CC$. A \defstyle{vector bundle of rank $\boldsymbol{n}$ over $\boldsymbol{M}$ with values in $\boldsymbol{\KK}$} is a topological manifold $E$ together with a surjective continuous map $\pi: E \to M$ such that for every point $p \in M$ there is an open neighborhood $U \subset M$ of $p$ with
	\begin{myenum}
		\item each \defstyle{fiber} $E_q := \pi^{-1}(q)$ at $q \in U$ is endowed with the vector space structure of $\KK^n$,
		\item there is a homeomorphism ${\psi : \restrict{E}{U} \to U \times \KK^n}$ from the restriction $\restrict{E}{U}:=\pi^{-1}(U)$ onto $U \times \KK^n$, called \defstyle{local trivialization}, such that
		\begin{myenum}
			\item the following diagram commutes:
			\begin{equation*}
			\begin{tikzcd}[column sep=tiny]
			\restrict{E}{U} \arrow["\pi", swap, dr] \arrow[rr, "\psi"] & & U \times \KK^n \arrow[dl, "\proj_1"] \\
			& U
			\end{tikzcd}
			\end{equation*}
			\item for each $q \in U$, the map $v \mapsto \psi^{-1}(q,v)$ is a vector space isomorphism between $\KK^n$ and $E_q$.
		\end{myenum}
	\end{myenum}
\end{mydef}
The space $M$ is called \defstyle{base space}. In contrast to that, the space $E$ is termed \defstyle{total space}. If both spaces are smooth manifolds, $\pi: E \to M$ is assumed to be a smooth submersion and all local trivializations are required to be diffeomorphisms, one speaks of a \defstyle{smooth} or \defstyle{differentiable} vector bundle. According to the context, we write \emph{\enquote{$E$ is a vector bundle over $M$}}, \emph{\enquote{$E \to M$ is a vector bundle}} or \emph{\enquote{$\pi:E \to M$ is a vector bundle}}. Notice that most operations applicable to vector spaces can be extended to vector bundles by performing the corresponding operation fiberwise. This particularly includes duals, exterior products, tensor products and direct sums.

Now suppose that $\psi_1,\psi_2$ are two local trivializations over $U_1$ and $U_2$ respectively with a non-empty overlap. Since every local trivialization defines an isomorphism fiberwise, the composite function
\begin{align*}
	\psi_2 \circ \psi^{-1}_1 : U_1 \cap U_2 \times \KK^n \to U_1 \cap U_2 \times \KK^n ,
\end{align*}
called \defstyle{transition map}, must be of the form
\begin{align}\label{eq:vec:transition}
	\left(\psi_2 \circ \psi^{-1}_1\right)(p,v) = \left(p, g(p)\,v\right) ,
\end{align}
where $g: U_1 \cap U_2 \to \mathrm{GL}(n, \KK)$ is a $C^r$-function. The function $g$ is called the \defstyle{transition function}. Here and in the following, we fix $r \in \{ 0, \infty \}$ depending on whether the bundle is topological or smooth.

\paragraph{Sections.}
Suppose $\pi : E \to M$ is a vector bundle of rank $n$. A section $\sigma$ on $E$ is a $C^r$-function, mapping each point $p$ on $M$ to its fiber $E_p$, i.\,e.\ $\pi \circ \sigma = \id_M$. Intuitively, in terms of trivializations, a section locally looks like a graph. For an open subset $U \subset M$, we henceforth define the \defstyle{space of $\boldsymbol{C^r}$-sections on $\boldsymbol{U}$} to be:
\begin{align*}
	C^r(U;E) := \left\{ \sigma: U \to E~\mbox{of class}~C^r ~|~ \sigma(p)\in E_p ~\mbox{for each}~p \in U \right\}.
\end{align*}
This space is naturally endowed with two algebraic structures. First of all, it is easy to see that $C^r(U;E)$ is a vector space over $\KK$ of infinite dimension. More interestingly, the space $C^r(U;E)$ is also a \emph{module} over the space of scalar fields $C^r(U,\KK)$ via pointwise multiplication.\footnote{More precisely, the sheaf of sections on $M$ is an $\mathcal{O}_M$-module, see~\cite[Ch.~13]{vakil:alggeo}.} Modules behave very similar to vector spaces, except that they have coefficients in arbitrary rings, here in the commutative ring $C^r(U,\KK)$, instead of in a field.

An important subcase is devoted to $C^r(U;E)$ being a \emph{free module}, which means that it is generated by a linearly independent set in terms of a module. We call such a set \defstyle{local frame}. Equivalently, a local frame can be seen as a set of sections forming a basis in every fiber $E_p$ at each point $p \in U$. This turns $C^r(U;E)$ into a free module of rank $n$. A vector bundle $E \to M$ whose global section module $C^r(M;E)$ is free is called \defstyle{trivial}.

If $U$ is a trivialization domain, we obtain a local frame very easily. Say $(e_1, \ldots, e_n)$ is the standard basis of $\KK^n$ and $\psi: \restrict{E}{U} \to U \times \KK^n$ is a trivialization. Then the map $p \mapsto \psi^{-1}(p,e_i)$ defines a $C^r$-section on $U$ for each $i=1,\ldots,n$. These maps assemble a local frame on $U$ as they constitute a basis in each fiber. Conversely, if $C^r(U;E)$ has a local frame $(\sigma_1, \ldots, \sigma_n)$, the bundle admits a trivialization over $U$ by setting $\tilde{\psi}\left(p, (v_1, \ldots, v_n)\right) = \sum^n_{i=1} v_i \, \sigma_i(p)$, which yields the desired $C^r$-diffeomorphism between $U \times \KK^n$ and $\restrict{E}{U}$.

Evidently by definition, local frames can be used to describe and determine arbitrary sections. For this, let $\{U_\alpha\}^\infty_{\alpha=1}$ be an open cover of $M$ such that each restriction $\restrict{E}{U_\alpha}$ is trivial.\footnote{In fact, it is always enough to consider a \emph{finite} cover. This is of course obvious for compact manifolds. The general case, however, can be proven in a similar way as to show that each topological manifold admits a finite atlas. Consult~\cite[Lem.~7.1, p.~77]{walschap:metric_structures} for details.} We fix a local frame $(e^\alpha_1, \ldots, e^\alpha_n)$ for each such domain $U_\alpha$. Then any global section $\sigma \in C^r(M;E)$ can be uniquely decomposed on each $U_\alpha$ in the following way:
\begin{align}\label{eq:vec:sec_decomp}
	\restrict{\sigma}{U_\alpha} = \sum^n_{i=1} f_\alpha^i \, e^\alpha_i, \quad \mbox{where}~f_\alpha^i \in C^r(U_\alpha, \KK).
\end{align}
This whole structure shows up as an useful model to implement sections into computer algebra, see Section~\ref{sec:vec:impl} for details.

A special kind of section is given by a \defstyle{bundle metric}. It is defined as a section $h$ in $C^r(M;E^* \otimes E^*)$ such that $h(p)$ defines a scalar product on the vector space $E_p$ at each point $p\in M$. Notice that every vector bundle, no matter whether real or complex, admits a bundle metric.

\paragraph{Tensor Bundles.}
In this paragraph, we discuss the most basic case of vector bundles: tensor bundles. Let $N$ be an $n$-dimensional smooth manifold and $(k, l) \in \NN^2$. We fix a point $q \in N$. A multilinear map
\begin{align*}
t:\ 	\underbrace{T_q^{\,*}N\times\cdots\times T_q^{\,*}N}_{k-\mbox{times}} \times \underbrace{T_q N\times\cdots\times T_q N}_{l-\mbox{times}} \longrightarrow \RR
\end{align*}
is called \defstyle{$\boldsymbol{(k,l)}$-tensor} with regards to the tangent space $T_q N$. The index $k$ denotes the \defstyle{contravariant} rank of $t$, while $l$ is its \defstyle{covariant} rank. If $k$ and $l$ equal zero, $t$ is simply a scalar. The corresponding space containing all $(k,l)$-tensors on $T_q N$ is written as $T_q^{\,(k,l)}N$. Likewise, we denote by
\begin{align*}
T^{\,(k,l)}N := \bigsqcup_{q \in N} T_q^{\,(k,l)}N
\end{align*}
the corresponding disjoint union on $N$. Now, let $M$ be another $m$-dimensional smooth manifold and $\varphi: M \to N$ be a smooth map. We define the \defstyle{tensor bundle of $\boldsymbol{(k,l)}$-tensors along $\boldsymbol\varphi$} to be the set
\begin{align*}
\varphi^* T^{\,(k,l)}N = \left\{ (p,t) \in M \times T^{\,(k,l)}N ~|~ t \in T_{\varphi(p)}^{\,(k,l)}N \right\}.
\end{align*}
By equipping the space $\varphi^* T^{\,(k,l)}N$ with the canonical footpoint map $\pi = \proj_1$, i.\,e.\ $(p,t) \mapsto p$, it inherits the structure of a vector bundle over $M$. This can be seen as follows. Let $(x^1, \ldots, x^n)$ be coordinates on an open subset $V \subset N$ such that $\varphi(p) = q \in V$ for a given $p \in M$. As usual, $( {\frac{\partial}{\partial x^1}}|_{q}, \ldots, {\frac{\partial}{\partial x^n}}|_{q})$ constitutes a basis on $T_q N$ and $( {\mathrm{d}x^1}|_{q}, \ldots, {\mathrm{d}x^n}|_{q})$ is its dual. Now, if we take a tensor $t \in T_{q}^{\,(k,l)}N$, its matrix entries with respect to that basis are given by
\begin{align*}
t^{a_1 \ldots a_k}_{\phantom{a_1 \ldots a_k}\, b_1 \ldots b_l} = t \left( \left.\frac{\partial}{\partial x^{a_1}}\right|_q, \dots, \left.\frac{\partial}{\partial x^{a_k}}\right|_q, \left.\mathrm{d}x^{b_1}\right|_q, \dots, \left.\mathrm{d}x^{b_l}\right|_q \right) \in \RR.
\end{align*}
These in turn induce a one-to-one correspondence between $U \times \RR^{n^{(k+l)}}$ and $\pi^{-1}(U)$, where we set $U = \varphi^{-1}(V)$:
\begin{align*}
(p,t) \mapsto \left(p, t^{1 \ldots 1}_{\phantom{1 \ldots 1}\, 1 \ldots 1}, \dots, t^{n \ldots n}_{\phantom{n \ldots n}\,n \ldots n} \right).
\end{align*}
A change of coordinates leads to an invertible linear transformation of the matrix entries. It is induced by the Jacobian matrix of the coordinate change and depends smoothly on the point $p$. Hence $\varphi^* T^{\,(k,l)}N$ is a smooth vector bundle over $M$ due to~\cite[Lem.~10.6]{lee:smooth_manifolds}. Incidentally, we observe that $\varphi^* T^{\,(k,l)}N$ is a smooth manifold of dimension $m + n^{(k+l)}$. A section of the tensor bundle $\varphi^* T^{\,(k,l)}N$ is called \defstyle{tensor field of type $\boldsymbol{(k,l)}$ along~$\boldsymbol\varphi$}.

The standard case of a tensor bundle over $M$ is given by $M=N$ and $\varphi=\id_M$. Common cases of tensor bundles over $M$ are the \emph{tangent bundle} $T^{\,(1,0)}M=TM$ and the \emph{cotangent bundle} $T^{\,(0,1)}M=T^{\,*}M$. If the tangent bundle is trivial, one says the manifold is \defstyle{parallelizable}.

\paragraph{Pullback Bundles.}
Motivated by the preceding paragraph, we want to generalize this construction to arbitrary vector bundles and obtain a new bundle under the presence of a map between manifolds. For this, let $M$ and $N$ be manifolds, $\pi:E\to N$ a vector bundle and $f:M \to N$ a function -- all of the same regularity. Out of that, we construct the following set:
\begin{align*}
f^* E := \{ (q,e) \in M \times E ~|~ f(q) = \pi(e) \}.
\end{align*}
Equipping $f^* E$ with the projection map $\pi':f^* E \to M$ given by the projection onto the first factor turns $f^* E$ into a vector bundle over $M$. Its local trivializations are induced by the ones of $E \to N$. Namely, if $\psi: \restrict{E}{U} \to U \times \KK^n$ is a local trivialization over $U\subset N$ then $\psi': \restrict{f^*E}{V} \to V \times \KK^n$ given by
\begin{align*}
	\psi'(q,e) = \left( q, \proj_2(\psi(e)) \right)
\end{align*}
is a local trivialization over $V = f^{-1}(U) \subset M$ on $f^*E$.

It is possible to transfer certain objects from $E$ to $f^*E$ via the pullback. For instance, given a section $\sigma \in C^r(M;E)$, we obtain a new section $f^* \sigma$ on $f^*E$ by composing it with~$f$, i.\,e.\ $$f^* \sigma = \sigma \circ f \in C^r(N;f^*E).$$ Similarly, most structures on vector bundles can be pulled back. This especially includes bundle metrics and bundle connections that we discuss in the proceeding~paragraphs.

\paragraph{Bundle Connections.}\label{par:vec:bundle_connections}
In this paragraph, we consider \emph{smooth} vector bundles. Depending on the bundle's base field, we possibly desire complex values on the cotangent bundle of its underlying manifold. We set $\KK=\RR,\CC$ and briefly denote $$T_{\KK}^{\,*}M = T^{\,*}M \otimes_{\RR} \KK.$$ By this notation, we mean that the tensor product is applied on each fiber. If $\KK$ is the complex field, we call this procedure \defstyle{complexification}. Even if nothing changes with $\KK$ being the real field, we keep this notation to clarify which field we are currently working on.\footnote{Moreover, one can examine differentiable manifolds and vector bundles over more general, non-discrete topological fields.} With this at hand, we start by declaring bundle connections.
\begin{mydef}\label{def:vec:bundle_connection}
	Suppose $E$ is a smooth (possibly complex) vector bundle over a manifold $M$. We choose $\KK$ to be $\RR$ or $\CC$ depending on whether $E$ is real or complex respectively. Then a \defstyle{bundle connection} is a $\KK$-linear map $$\nabla: C^\infty(M; E) \to C^\infty(M; E \otimes T_{\KK}^{\,*}M)$$ satisfying the Leibniz rule $$\nabla(f \cdot \sigma ) = \sigma \otimes \diffd f + f \cdot \nabla \sigma$$ for each section $s \in C^\infty(M;E)$ and $\KK$-valued scalar field $f \in C^\infty(M, \KK)$. If $X$ is a vector field on $M$, we further denote by $\nabla_X \sigma := (\nabla \sigma)(X)$ the \defstyle{covariant derivative of $\boldsymbol{\sigma}$ along $\boldsymbol{X}$}.
\end{mydef}
One of the basic properties of a connection is that it is a local operator and decreases support. Meaning, if the section $\sigma$ has its support on an open subset $U\subset M$ then $\nabla \sigma$ is supported on $U$ as well. In this way, it makes sense to restrict $\nabla$ to \emph{local} sections. Finally, notice that every smooth vector bundle possesses a bundle connection, compare \cite[Lem.~2,~p.~291]{milnor:char_classes}.

There is one particular class of connections that is quite important to us. Suppose $E$ is equipped with a bundle metric $\scal{\,\cdot\,}{\cdot\,}$. Then a bundle connection $\nabla$ on $E$ is called \defstyle{metric} or \defstyle{compatible} with respect to  $\scal{\,\cdot\,}{\cdot\,}$ iff
\begin{align*}
	\partial_X \scal{\sigma_1}{\sigma_2} = \scal{\nabla_X \sigma_1}{\sigma_2} + \scal{\sigma_1}{\nabla_X \sigma_2}
\end{align*}
holds for any $\sigma_1,\sigma_2 \in C^{\infty}(M;E)$ and $X \in C^\infty(M; TM)$.

For a given connection, the curvature can be obtained as follows. For each section $\sigma \in C^\infty(M;E)$ and each vector field $X,Y \in C^\infty(M;TM)$ we define the quantity
$$R(X,Y)\sigma = \nabla_X \nabla_Y \sigma - \nabla_Y \nabla_X \sigma - \nabla_{[X,Y]} \sigma .$$
This gives rise to an $\KK$-linear map
\begin{align*}
	R : C^\infty(M; E) \to C^\infty\!\left(M; E \otimes {\bigwedge}^2\, T_{\KK}^{\,*}M \right)
\end{align*}
which is called the \defstyle{curvature tensor}.

Now assume $E$ has rank $n$ and let $(e_1, \dots, e_n)$ be a local frame of $E$ on some open subset $U \subset M$. Then we can find 1-forms $\omega^j_i \in C^\infty(U ; T_{\KK}^{\,*}M)$ such that we can write
\begin{align}\label{eq:vec:def_con_form}
	\nabla e_i = \sum_{j=1}^{n} e_j \otimes \omega^{j}_i,
\end{align}
which induces an $(n \times n)$-matrix $\omega \in C^\infty\big(U ; \mathfrak{gl}(n,\KK) \otimes_{\RR} T^{\,*}M\big)$ called \defstyle{connection form matrix of $\boldsymbol\nabla$ with respect to $\boldsymbol e$}.\footnote{More specifically, $\omega$ is an $\mathfrak{gl}(n,\KK)$-valued form; compare Definition~\ref{def:mixed:valued_forms}.} As usual in the notion of Lie algebras, we denote $\mathfrak{gl}(n,\KK)={\mathrm{Mat}(n \times n, \CC)}$. Similarly, by substituting $\nabla$ with $R$, we obtain 2-forms $\Omega^{j}_i \in C^\infty(U ; {\bigwedge}^2\, T_{\KK}^{\,*}M)$ satisfying
\begin{align*}
	R \, e_i = \sum_{j=1}^{n} e_j \otimes \Omega^{j}_i.
\end{align*}
These give rise to an $(n \times n)$-matrix $\Omega \in C^\infty\big(U ; \mathfrak{gl}(n,\KK) \otimes_{\RR} {\bigwedge}^2\, T^{\,*}M\big)$, called \defstyle{curvature form matrix of $\boldsymbol\nabla$ with respect to $\boldsymbol e$}. A straightforward computation reveals the following relation between $\Omega$ and $\omega$:
\begin{align}\label{eq:vec:curv_from_con}
\Omega_i^j = \diffd \omega_i^j + \sum^n_{k=1} \omega_k^j \wedge \omega_i^k .
\end{align}

We want to make some remarks at this point. First, if $E$ is endowed with a bundle metric and a compatible connection $\nabla$, the associated curvature form matrix is skew-Hermitian for each orthonormal local frame. Secondly, it is readily checked that $\Omega$ transforms by $g \Omega g^{-1}$ as expected under a change of framing~${g: U \to \mathrm{GL}_n(\KK)}$. Beware that this is not the case for $\omega$: if $\omega'$ is the connection form matrix with respect to the new frame then we have
\begin{align*}
	\omega' = g^{-1} \diffd g + g^{-1} \omega g.
\end{align*}

\subsection{SageMath Implementation}\label{sec:vec:impl}
Vector bundles are an immediate generalization of the tangent bundle over differentiable manifolds. It is therefore not surprising that the preexisting code for symbolic tensor calculus can be used as a reference point for implementing vector bundles. In particular, each of our classes around vector bundles is based on a preexisting class taken as role model. Table~\ref{tab:vec:ingredients} displays all new classes implemented in \Sage. The right column shows the class on which it is based. In this section, we highlight the most important classes and its implementation concepts. A full list of supported features can be gathered from \Sage's reference manual~\cite{sagemath:reference:manif}.

\begin{table}[t]
	\centering
	\footnotesize
	\begin{tabular}{lcr}
		\toprule
		\textbf{Mathematical Object} & \textbf{Represented by} & \textbf{In Contrast to} \\
		\midrule
		Vector bundle & \texttt{TopologicalVectorBundle} & \texttt{DifferentiableManifold} \\
		$E \to M$ & \texttt{DifferentiableVectorBundle} & \\ \arrayrulecolor{darkgray}\hline
		Trivialization & \texttt{Trivialization} & \texttt{DiffChart} \\
		$\varphi:\restrict{E}{U} \to U \times \KK^n$ & & \\ \arrayrulecolor{darkgray}\hline
		Local frame & \texttt{LocalFrame} & \texttt{VectorFrame} \\
		$\left( \restrict{E}{U}, (e_1, \ldots, e_n) \right)$ & & \\ \arrayrulecolor{darkgray}\hline
		Section & \texttt{TrivialSection} & \texttt{TensorFieldParal} \\
		$s \in C^r(U;E)$ & \texttt{Section} & \texttt{TensorField} \\ \arrayrulecolor{darkgray}\hline
		Section module & \texttt{SectionFreeModule} & \texttt{TensorFieldFreeModule} \\
		$C^r(U;E)$ & \texttt{SectionModule} & \texttt{TensorFieldModule} \\ \arrayrulecolor{darkgray}\hline
		Fiber & \texttt{VectorBundleFiber} & \texttt{TangentSpace} \\
		$E_p$ & & \\ \arrayrulecolor{darkgray}\hline
		Fiber element & \texttt{VectorBundleFiberElement} & \texttt{TangentVector} \\
		$v \in E_p$ & & \\ \arrayrulecolor{darkgray}\hline
		Bundle connection & \texttt{BundleConnection} & \texttt{AffineConnection} \\
		$\nabla^E$ & & \\
		\bottomrule
	\end{tabular}
	\caption{A table of all ingredients necessary to realize vector bundles in \Sage. In the right column we see its analogue to the preexisting implementation of manifolds.}\label{tab:vec:ingredients}
\end{table}

\paragraph{Vector Bundles.}
Included in the \texttt{SageManifolds} package, differentiable manifolds are represented by instances of the class \texttt{DifferentiableManifold} which act like a \enquote{control center} to everything that is associated to them. They store all necessary information and communicate them to the subordinate structures. Providing a similar attempt, vector bundles are implemented via \texttt{TopologicalVectorBundle} belonging to the category \texttt{VectorBundles}. An instance of this class is uniquely determined by the base space $M$, the rank $n$, the underlying field and the total space's name $E$. Special cases like differentiable vector bundles and tensor bundles are inherited from this particular class. The full inheritance tree can be found in Figure~\ref{fig:vec:inheritance_vbundle}. Notice that tensor bundles completely fall back on the preexisting implementation of tensor fields.

\begin{figure}[t]
	\centering
	\tikzstyle{arrow} = [thick,->,>=stealth]
	\tikzstyle{factory} = [rectangle, rounded corners, minimum width=2.5cm, minimum height=.75cm,text centered, draw=black, fill=strcolor!50, thin]
	\tikzstyle{obj} = [rectangle, rounded corners, minimum width=2.5cm, minimum height=.75cm,text centered, draw=black, fill=gray!30, thin]
	\begin{tikzpicture}[node distance=\textwidth / 10, >=latex', thick]
	\node (top-vecb) [factory, align=center] {\footnotesize\texttt{TopologicalVectorBundle}};
	\node (diff-vecb) [factory, align=center, below of=top-vecb] {\footnotesize\texttt{DifferentiableVectorBundle}};
	\node (ten-vecb) [factory, align=center, below of=diff-vecb] {\footnotesize\texttt{TensorBundle}};
	
	\node (category) [obj, align=center, above of=top-vecb,xshift=-2.5cm] {\footnotesize\texttt{CategoryObject}};
	\node (unique) [obj, align=center, above of=top-vecb,xshift=2.5cm] {\footnotesize\texttt{UniqueRepresentation}};
	
	\draw [arrow] (diff-vecb) -- (top-vecb) node[midway,right,darkgray] {};
	\draw [arrow] (ten-vecb) -- (diff-vecb) node[midway,right,darkgray] {};
	\draw [arrow] (top-vecb) -- (category) node[midway,left,darkgray] {};
	\draw [arrow] (top-vecb) -- (unique) node[midway,right,darkgray] {};
	\end{tikzpicture}
	\caption{The inheritance diagram for the representative classes of vector bundles.}\label{fig:vec:inheritance_vbundle}
\end{figure}
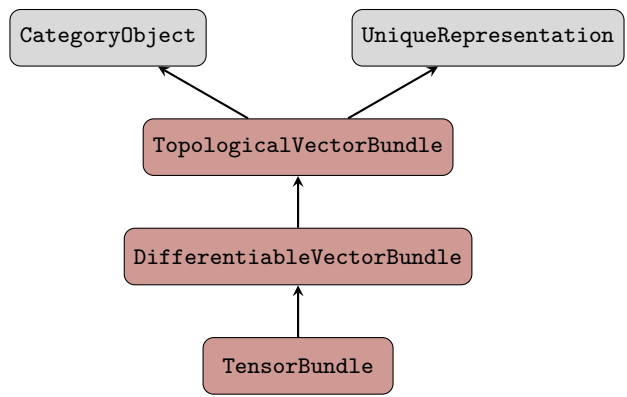

Figure~\ref{fig:vec:control_center} gives an insight of how vector bundles are realized to facilitate the interaction of the structures. Most objects can be obtained by invoking proper methods on the vector bundle's instance and are communicated back once they are created. This supplies each object in correspondence to the vector bundle with all the necessary information to fulfill its purpose.

\paragraph{Sections.}\label{par:vec:impl_sections}
For an actual realization, it is useful to see the space of sections $C^r(U;E)$ as a module over $C^r(U, \KK)$ rather than a vector space over $\KK$. The reason is simple: in terms of a vector space, $C^r(U;E)$ is infinite dimensional which is difficult to accomplish in computer algebra. In contrast, considering $C^r(U;E)$ as a module over $C^r(U, \KK)$, which has finite rank if it is free, means implementing a \emph{finite} structure. The implementation splits in two parts: an \emph{algebraic part} carrying out free modules of finite rank over arbitrary commutative rings, and a \emph{geometric part} patching everything together over the manifold. Both have already been implemented during the \texttt{SageManifolds} project but specialized for tensor fields. The approach for sections and its modules for vector bundles follows exactly the same idea, generalizing the preexisting implementation of tensor fields and relying on the available algebraic part. In a nutshell: each section module having a local frame is free and hence realized by exploiting the algebraic part. All other modules are assumed to be non-free. Their sections are patched together from restrictions living in the free modules due to~\eqref{eq:vec:sec_decomp}. We refer to the \Sage reference manual~\cite{sagemath:reference:manif, sagemath:reference:free_tensor} for more details. The initial implementation of tensor fields in \Sage is thoroughly discussed in~\cite[Sec.~4.4~et~seq.]{gour:sagemanifolds}.

\begin{figure}[t]
	\centering
	\tikzstyle{arrow} = [thick,->,>=stealth]
	\tikzstyle{factory} = [rectangle, rounded corners, minimum width=3cm, minimum height=1cm,text centered, draw=black, fill=strcolor!50, thin]
	\tikzstyle{obj} = [rectangle, rounded corners, minimum width=3cm, minimum height=1cm,text centered, draw=black, fill=gray!30, thin]
	\begin{tikzpicture}[node distance=\textwidth / 8, >=latex', thick]
	\node (vbundle) [factory, align=center] {{\scriptsize Vector bundle:} \\ $\displaystyle E \to M$};
	\node (triv) [obj, align=center, above of=vbundle] {{\scriptsize Trivialization:} \\ \scriptsize $\varphi: \restrict{E}{U} \to U \times \KK^n$};
	\node (frame) [obj, align=center, right of=vbundle,xshift=3.25cm] {{\scriptsize Local frame:} \\ \scriptsize $\left(\restrict{E}{U}, (e_1, \ldots, e_n)\right)$};
	\node (sec-mod) [obj, align=center, left of=vbundle,xshift=-3.25cm] {{\scriptsize Section module:} \\ $C^r(U;E)$};
	\node (sec) [obj, align=center, below of=vbundle] {{\scriptsize Section:} \\ $s \in C^r(U;E)$};
	
	\node (legend) [left of=triv] {};
	
	\draw [arrow] ([xshift=.6ex]vbundle.north) -- ([xshift=.6ex]triv.south) node[midway,right] {\tiny\texttt{trivialization}};
	\draw [arrow, dashed, darkgray] ([xshift=-.6ex]triv.south) -- ([xshift=-.6ex]vbundle.north) node[midway,left,darkgray] {\tiny stored in};
	
	\draw [arrow] ([yshift=.6ex]vbundle.east) -- ([yshift=.6ex]frame.west) node[midway,above] {\tiny\texttt{local\_frame}};
	\draw [arrow,dashed, darkgray] ([yshift=-.6ex]frame.west) -- ([yshift=-.6ex]vbundle.east) node[midway,below,darkgray] {\tiny stored in};
	
	\draw [arrow] ([yshift=.6ex]vbundle.west) -- ([yshift=.6ex]sec-mod.east) node[midway,above] {\tiny\texttt{section\_module}};
	\draw [arrow,dashed, darkgray] ([yshift=-.6ex]sec-mod.east) -- ([yshift=-.6ex]vbundle.west) node[midway,below,darkgray] {\tiny stored in};
	
	\draw [arrow] (vbundle) -- (sec) node[midway,right] {\tiny\texttt{section}};
	\draw [arrow] (triv) -| (frame) node[midway,above,xshift=-1.7cm] {\scriptsize\texttt{frame}};
	\draw [arrow] (frame) |- (sec) node[midway,below,xshift=-1.7cm] {\scriptsize\texttt{\_\_getitem\_\_}};
	\draw [arrow, dashed, darkgray] (sec-mod) |- (sec) node[midway,below,darkgray,xshift=1.7cm] {\scriptsize parent of};

	\end{tikzpicture}
	\caption{Our realization of a vector bundle plays the role of a \enquote{control center} to connect all subordinate structures.}\label{fig:vec:control_center}
\end{figure}
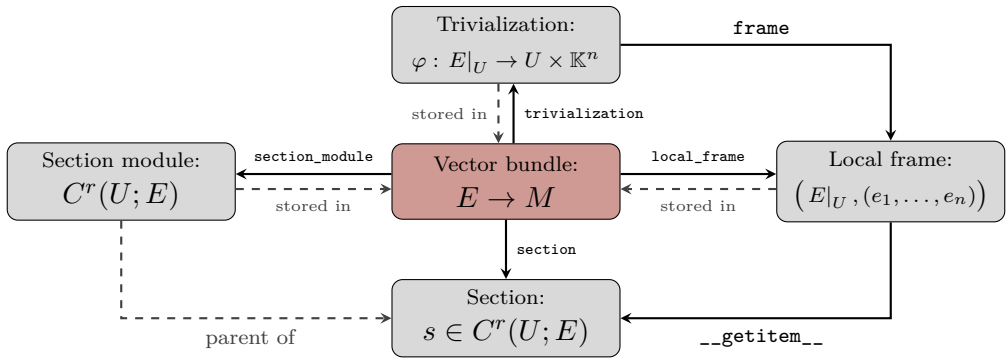

\paragraph{Trivializations and Frames.}
The \Python class \texttt{LocalFrame}  represents local frames within \Sage. It inherits from the algebraic counterpart \texttt{FreeModuleBasis}. Once a local frame on an open subset $U \subset M$ is initialized, \Sage acknowledges the corresponding section module $C^r(U ; E)$ as free. From then on, all sections can be expressed locally in this frame. Of course, a local frame itself consists of local sections. They can be returned by simple index operations applied to the object. A change of frame is performed by invoking the vector bundle's method \texttt{set\_change\_of\_frame} taking automorphisms on the corresponding free module. These automorphisms are represented by instances of \texttt{FreeModuleAutomorphism}. This class belongs to the algebraic implementation.

Even though local frames are the key concept here, sometimes it is more intuitive to think in terms of local trivializations instead. For this reason, our implementation offers the additional class \texttt{Trivialization}. It behaves a lot like \texttt{DiffChart} whose instances are differentiable charts. As well as for charts, the transition map can be stated by invoking \texttt{transition\_map} and inserting the desired transformation. In contrast to charts, the transformation is given by an $(n \times n)$-matrix consisting of scalar fields on the overlap, compare~\eqref{eq:vec:transition}. \Sage immediately translates the input into the language of frames and announces the change of frame to the vector bundle.

\paragraph{Bundle Connections.}
At this stage, bundle connections can only be used to assign connection forms and compute curvature forms with respect to a given local frame. The corresponding formula is given in~\eqref{eq:vec:curv_from_con}. They are realized via the \Python class \texttt{BundleConnection}. Each instance stores a \Python dictionary \texttt{\_connection\_forms} whose keys are local frames. To each key in \texttt{\_connection\_forms} corresponds another \Python dictionary obeying the allocation rule $(i,j) \mapsto \omega^i_j$ with regards to the defining equation~(\ref{eq:vec:def_con_form}). A more elaborated version is already on the agenda, see Section~\ref{sec:perspective}. Examples on the usage of bundle connections in \Sage are discussed along with examples of characteristic classes in Chapter~\ref{ch:char_class}.

\subsection{Example: The Möbius Bundle}\label{sec:vec:moebius}
We want to introduce a non-trivial line bundle and explain how it can be applied within \Sage making use of our implementation discussed in the previous section. First of all, we declare an equivalence relation on $\RR^2 \setminus \{ 0\}$:
\begin{align*}
	(x,y) \sim (x',y') \quad :\Longleftrightarrow \quad \exists \lambda \in \RR: (x,y) = \lambda\, (x', y').
\end{align*}
We define the \defstyle{real-projective space} by taking the corresponding quotient:
\begin{align*}
	\mathbb{RP}^1 := \faktor{(\RR^2 \setminus \{0\})}{\sim}.
\end{align*}
This space is a 1-dimensional topological manifold~\cite[Example~1.5]{lee:smooth_manifolds} and canonically endowed with \defstyle{homogeneous coordinates} given by the abbreviation ${[x:y] = \left[ (x,y) \right]_\sim}$. Then we can write
\begin{align*}
	\mathbb{RP}^1 = \left\{ [x:y] ~|~ (x,y) \in \RR^2 \setminus \{ 0\} \right\}.
\end{align*}
Geometrically, the real-projective space corresponds to all 1-dimensional subspaces of $\RR^2$. In order to obtain suitable charts on $\mathbb{RP}^1$, we define the two subsets
\begin{align*}
	U := \left\{ [1:u] ~|~ u \in \RR \right\} \quad \text{and} \quad V := \left\{ [v:1] ~|~ v \in \RR \right\}
\end{align*}
evidently covering $\mathbb{RP}^1$.
\begin{figure}[t]
	\centering
	\begin{tikzpicture}
	\draw [->,thick,>=stealth'] (0,-1.5) -- (0,3) node (yaxis) [above] {$y$};
	\draw [->,thick,>=stealth'] (-3,0) -- (5,0) node (xaxis) [right] {$x$};
	
	\fill [opacity=.5,thick] (2,0) circle (1.5pt) node[anchor=north east] {\small$(1,0)$};
	\fill [opacity=.5,thick] (0,2) circle (1.5pt) node[anchor=north east] {\small$(0,1)$};
	
	\draw [opacity=.9, kwcolor, thick] (4.2,2.4) coordinate (l1) -- (-2.52,-1.44) coordinate (l2) node [above] (line) {$\ell$};
	
	\fill [opacity=.9,thick] (0,0) circle (1.5pt) node[anchor=south east] {\small$(0,0)$};
	
	\draw [opacity=.5, dashed, thick] (4.75,2) coordinate (v1) -- (-3,2) coordinate (v2) node [left] (v-line) {};
	\draw [opacity=.5, dashed, thick] (2,2.75) coordinate (u1) -- (2,-1.5) coordinate (u2) node [below] (u-line) {};
	
	\coordinate (lv) at (intersection of v1--v2 and l1--l2);
	\coordinate (lu) at (intersection of u1--u2 and l1--l2);
	
	\fill[opacity=.9,kwcolor] (lv) circle (1.5pt) node[anchor=north west] {\small$(v,1)$};
	\fill[opacity=.9,kwcolor] (lu) circle (1.5pt) node[anchor=south east] {\small$(1,u)$};
	\end{tikzpicture}
	\caption{Charts on $\mathbb{RP}^1$ can be obtained by intersections with affine hyperplanes $x=1$ and $y=1$.}\label{fig:vec:projective_space}
\end{figure}
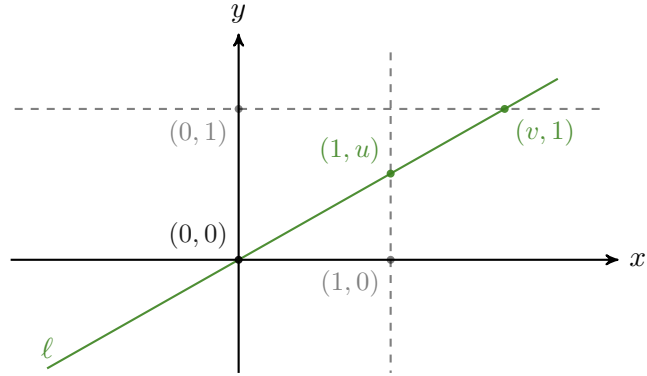
As pictured in Figure~\ref{fig:vec:projective_space}, the set $U$ corresponds to all lines through the origin intersecting the hyperplane $x=1$. Similarly, $V$ consists of all lines intersecting the hyperplane $y=1$. It is readily checked that the following two maps define homeomorphisms:
\begin{align*}
	h_U: U \to \RR  &\quad\text{with}\quad [1:u] \mapsto u, \\
	h_V: V \to \RR  &\quad\text{with}\quad [v:1] \mapsto v.
\end{align*}
The corresponding change of coordinates $$h_V \circ h_U^{-1} : \RR \setminus \{ 0\} \to \RR \setminus \{ 0\}$$ on the intersection domain $W = U \cap V$ can be extracted from Figure~\ref{fig:vec:projective_space} and is determined as $u \mapsto \frac{1}{u}$. In contrast to the homogeneous coordinates above, the ones given by $h_U$ and $h_V$ are called \defstyle{inhomogeneous coordinates}.

Of course, this construction can be generalized to higher dimensions and complex numbers. However, we want to keep things simple for now and apply this setup to \Sage: 
\begin{NBin}
M = Manifold(1, 'RP1', 
			 latex_name=r'\mathbb{RP}^1', 
			 start_index=1,
			 structure='topological')
U = M.open_subset('U'); hu.<u> = U.chart()
V = M.open_subset('V'); hv.<v> = V.chart()
M.declare_union(U, V)
\end{NBin}
The change of coordinates, as discussed above, is given by:
\begin{NBin}
u_to_v = hu.transition_map(hv, 1/u, 
						   intersection_name='W',
						   restrictions1= u!=0,
						   restrictions2= v!=0)
u_to_v.display()
\end{NBin}
\begin{NBoutM}
$\begin{array}{lcl} v & = & \frac{1}{u} \end{array}$
\end{NBoutM}
\begin{NBin}
W = U.intersection(V)
v_to_u = u_to_v.inverse()
v_to_u.display()
\end{NBin}
\begin{NBoutM}
$\begin{array}{lcl} u & = & \frac{1}{v} \end{array}$
\end{NBoutM}
For this example, we investigate the so-called \defstyle{Möbius bundle}
\begin{align*}
E = \left\{ (\ell, v) \in \mathbb{RP}^1 \times \RR^2 ~|~ v \in \ell \cup \{0\} \right\},
\end{align*}
endowed with the projection $(\ell, v) \mapsto \ell$. To obtain a local trivialization on $E$, let $(\ell, v)$ be in $E$ such that $\ell \in U$. Then we can write $\ell = [1:u]$ for some $u \in \RR$, and the vector $(1,u) \in \RR^2$ forms a basis of the 1-dimensional subspace $\ell \cup \{ 0\}$. Since $v$ is a vector in that space, we can find a unique $\lambda_{\ell, v} \in \RR$ such that $v = \lambda_{\ell,v} \, (1,u)$. This gives rise to a local trivialization $\psi_U: \restrict{E}{U} \to U \times \RR$ via $\psi_U(\ell, v) = (\ell, \lambda_{\ell,v})$. The construction of $\psi_V: \restrict{E}{U} \to V \times \RR$ over $V$ is performed analogously. Taking a second look at Figure~\ref{fig:vec:projective_space}, we derive that the transition map $\psi_V \circ \psi^{-1}_U$ is induced by the transition function $W \to \mathrm{GL}(1,\RR)$ via $[1:u] \mapsto (u)$. %
With this, we are now ready to define this vector bundle in \Sage:
\begin{NBin}
E = M.vector_bundle(1, 'E'); print(E)
\end{NBin}
\begin{NBout}
Topological real vector bundle E -> RP1 of rank 1 over the base space 1-dimensional topological manifold RP1
\end{NBout}
Let us state the two trivializations constructed above:
\begin{NBin}
psiU = E.trivialization('psiU', latex_name=r'\psi_U', domain=U)
psiU
\end{NBin}
\begin{NBoutM}
$\psi_U : E |_{U} \to U \times \Bold{R}^1$
\end{NBoutM}
\begin{NBin}
psiV = E.trivialization('psiV', latex_name=r'\psi_V', domain=V)
psiV
\end{NBin}
\begin{NBoutM}
$\psi_V : E |_{V} \to V \times \Bold{R}^1$
\end{NBoutM}
Next we declare the transition map between $\psi_U$ and $\psi_V$:
\begin{NBin}
transf = psiU.transition_map(psiV, [[u]]); transf
\end{NBin}
\begin{NBoutM}
$\psi_V\circ \psi_U^{-1}:W\times \Bold{R}^1 \to W\times \Bold{R}^1$
\end{NBoutM}
Each trivialization induces a local frame which we can get by the following command:
\begin{NBin}
eU = psiU.frame(); eU
\end{NBin}
\begin{NBoutM}
$\left(E|_{U}, \left(\left(\psi_U^* e_{ 1 }\right)\right)\right)$
\end{NBoutM}
The notation $\left(\psi_U^* e_{ 1 }\right)$ stands for the local section on $U$ given by $p \mapsto \psi^{-1}_U(p,e_1)$, where $e_1$ is the standard basis of $\RR$. Similarly, we have for $\psi_V$:
\begin{NBin}
eV = psiV.frame(); eV
\end{NBin}
\begin{NBoutM}
$\left(E|_{V}, \left(\left(\psi_V^* e_{ 1 }\right)\right)\right)$
\end{NBoutM}
The corresponding bundle automorphism $\psi_U^{-1} \circ \psi_V:\restrict{E}{W} \to \restrict{E}{W}$ translating $\left(\psi_V^* e_{ 1 }\right)$ into $\left(\psi_U^* e_{ 1 }\right)$ can be easily returned:
\begin{NBin}
transf.automorphism()
\end{NBin}
\begin{NBoutM}
$\psi_U^{-1} \circ \psi_V$
\end{NBoutM}
As discussed in Section~\ref{sec:vec:impl}, this is an instance of \texttt{FreeModuleAutomorphism}:
\begin{NBin}
from sage.tensor.modules.free_module_automorphism \
											import FreeModuleAutomorphism
isinstance(transf.automorphism(), FreeModuleAutomorphism)
\end{NBin}
\begin{NBoutM}
$\mathrm{True}$
\end{NBoutM}
We can even get its determinant which is a scalar field on the intersection $W = U \cap V$:
\begin{NBin}
transf.det().display()
\end{NBin}
\begin{NBoutM}
$\begin{array}{llcl} \det(\psi_U^{-1} \circ \psi_V):& W & \longrightarrow & \mathbb{R} \\ & u & \longmapsto & u \\ & v & \longmapsto & \frac{1}{v} \end{array}$
\end{NBoutM}
We can see that the determinant is negative if $u<0$. It is therefore reasonable to suspect that the vector bundle $E$ is not orientable. This is indeed true as proven in~\cite[pp.~16--17]{milnor:char_classes}. As a consequence, $E$ is not trivial and each global section must vanish somewhere. To illustrate this, we define the corresponding section module over $\mathbb{RP}^1$:
\begin{NBin}
C0 = E.section_module(); C0
\end{NBin}
\begin{NBoutM}
$C^{0}(\mathbb{RP}^1;E)$
\end{NBoutM}
We can see that \Sage rejects $C^{0}(\mathbb{RP}^1;E)$ as a free module:
\begin{NBin}
from sage.tensor.modules.finite_rank_free_module import FiniteRankFreeModule
isinstance(C0, FiniteRankFreeModule)
\end{NBin}
\begin{NBoutM}
$\mathrm{False}$
\end{NBoutM}
This is because there is no global frame that \Sage knows about:
\begin{NBin}
E.is_manifestly_trivial()
\end{NBin}
\begin{NBoutM}
$\mathrm{False}$
\end{NBoutM}
On the contrary, the section module over $U$ must be free:
\begin{NBin}
C0U = E.section_module(domain=U); C0U
\end{NBin}
\begin{NBoutM}
$C^0(U;E)$
\end{NBoutM}
And indeed, it is:
\begin{NBin}
from sage.tensor.modules.finite_rank_free_module import FiniteRankFreeModule
isinstance(C0U, FiniteRankFreeModule)
\end{NBin}
\begin{NBoutM}
$\mathrm{True}$
\end{NBoutM}
We start with some concrete computations and therefore define a section on $U$:
\begin{NBin}
sU = E.section(name='sigma', latex_name=r'\sigma', domain=U)
sU[eU,1] = (1-u)/(1+u^2)
sU.display()
\end{NBin}
\begin{NBoutM}
$\sigma = \left( -\frac{u - 1}{u^{2} + 1} \right) \left(\psi_U^* e_{ 1 }\right)$
\end{NBoutM}
This local section lives in the free module $C^0(U;E)$:
\begin{NBin}
sU in C0U
\end{NBin}
\begin{NBoutM}
$\mathrm{True}$
\end{NBoutM}
We can perform a change of frame on the subset $W$:
\begin{NBin}
sU.display(eV.restrict(W), hv.restrict(W))
\end{NBin}
\begin{NBoutM}
$\sigma = \left( \frac{v - 1}{v^{2} + 1} \right) \left(\psi_V^* e_{ 1 }\right)$
\end{NBoutM}
This expression is obviously well-defined on the whole subset $V$. Hence, we can extend this section continuously onto $\mathbb{RP}^1$:
\begin{NBin}
s = E.section(name='sigma', latex_name=r'\sigma')
s.set_restriction(sU)
s.add_comp_by_continuation(eV, W)
s.display(eV)
\end{NBin}
\begin{NBoutM}
$\sigma = \left( \frac{v - 1}{v^{2} + 1} \right) \left(\psi_V^* e_{ 1 }\right)$
\end{NBoutM}
The corresponding continuation is indeed an element of $C^{0}(\mathbb{RP}^1;E)$:
\begin{NBin}
s in C0
\end{NBin}
\begin{NBoutM}
$\mathrm{True}$
\end{NBoutM}
Let us define another global section in $E$:
\begin{NBin}
t = E.section(name='tau', latex_name=r'\tau')
t[eV,1] = (3-v^2)/(1+v^4)
t.add_comp_by_continuation(eU, W)
t.display(eU)
\end{NBin}
\begin{NBoutM}
$\tau = \left( \frac{3 \, u^{3} - u}{u^{4} + 1} \right) \left(\psi_U^* e_{ 1 }\right)$
\end{NBoutM}
Now, $\sigma$ and $\tau$ can be added pointwise:
\begin{NBin}
r = (s + t); r.display(eU)
\end{NBin}
\begin{NBoutM}
$\sigma+\tau = \left( \frac{2 \, u^{5} + u^{4} + 2 \, u^{3} - 2 \, u + 1}{u^{6} + u^{4} + u^{2} + 1} \right) \left(\psi_U^* e_{ 1 }\right)$
\end{NBoutM}
\begin{NBin}
r.display(eV)
\end{NBin}
\begin{NBoutM}
$\sigma+\tau = \left( \frac{v^{5} - 2 \, v^{4} + 2 \, v^{2} + v + 2}{v^{6} + v^{4} + v^{2} + 1} \right) \left(\psi_V^* e_{ 1 }\right)$
\end{NBoutM}
Since $\sigma + \tau$ is again a well-defined continuous section on $E$, it must vanish at some point. We want to check this by solving an equation:
\begin{NBin}
sol = solve(r[eU,1,hu].expr() == 0, u, solution_dict=True)
sol
\end{NBin}
\begin{NBoutM}
$\left[\left\{u : -1\right\}\right]$
\end{NBoutM}
Let us investigate what happens at this particular point $p\in \mathbb{RP}^1$ determined by $u=-1$:
\begin{NBin}
p = M.point([sol[0][u]], name='p', chart=hu); print(p)
\end{NBin}
\begin{NBout}
Point p on the 1-dimensional topological manifold RP1
\end{NBout}
The corresponding section $\sigma$ evaluated at $p$ is an element of the fiber $E_p$:
\begin{NBin}
print(s.at(p))
\end{NBin}
\begin{NBout}
Vector sigma in the fiber of E at Point p on the 1-dimensional topological manifold RP1
\end{NBout}
Concretely, we have: 
\begin{NBin}
s.at(p).display(basis=eU.at(p))
\end{NBin}
\begin{NBoutM}
$\sigma = \left(\psi_U^* e_{ 1 }\right)$
\end{NBoutM}
For $\tau$ we similarly obtain:
\begin{NBin}
t.at(p).display(basis=eU.at(p))
\end{NBin}
\begin{NBoutM}
$\tau = - \left(\psi_U^* e_{ 1 }\right)$
\end{NBoutM}
As expected, the sum vanishes at $p$:
\begin{NBin}
r.at(p).display(basis=eU.at(p))
\end{NBin}
\begin{NBoutM}
$\sigma+\tau = 0$
\end{NBoutM}

\subsection{Example: Tensor Bundles over $\Sphere^2$}
In this section, we want to highlight the main features of tensor bundles. The 2-sphere $\Sphere^2 \subset \RR^3$ serves as a suitable example for this. We define $U \subset \Sphere^2$ to be the complement of the meridian lying in the upper $x$-$z$-plane for $x\geq 0$. Similarly, $V \subset \Sphere^2$ defines the complement of the meridian going through the $x$-$y$-plane for $x \leq 0$. Both sets cover $\Sphere^2$ and are open in the relative topology. In the language of \Sage, we write:
\begin{NBin}
M = Manifold(2, name='S^2', latex_name=r'\mathbb{S}^2')
U = M.open_subset('U'); V = M.open_subset('V')
M.declare_union(U,V) # M is the union of U and V
\end{NBin}
The corresponding tangent bundle can be returned by the following command:
\begin{NBin}
TM = M.tangent_bundle(); TM
\end{NBin}
\begin{NBoutM}
$T\mathbb{S}^2\to \mathbb{S}^2$
\end{NBoutM}
For now, however, we are interested in the parallelizable subset $U \subset \Sphere^2$:
\begin{NBin}
TU = U.tangent_bundle(); print(TU)
\end{NBin}
\begin{NBout}
Tangent bundle TU over the Open subset U of the 2-dimensional differentiable manifold S^2
\end{NBout}
Each chart on a manifold gives rise to a trivialization on the corresponding tangent bundle. The converse, however, is false as the example in Section~\ref{sec:char:berger} indicates. Nevertheless, trivializations entirely fall back on the class \texttt{DiffChart} at this stage. The reason is to allow the preexisting implementation to do the whole work and facilitate applications.\footnote{If one still desires frames which are not induced by charts, this can be achieved by using the method \texttt{local\_frame} or \texttt{vector\_frame} respectively.} To demonstrate how it is done, we introduce spherical coordinates on the subset $U$:
\begin{NBin}
c_spher.<th,ph> = TU.trivialization(r'th:(0,pi):\theta ph:(0,2*pi):\phi')
\end{NBin}
To demonstrate the pullback of tensor bundles, we define the Euclidean space $\RR^3$ and introduce a differential map $\varphi: \Sphere^2 \to \RR^3$ given by the embedding of $\Sphere^2$ into $\RR^3$:
\begin{NBin}
R = Manifold(3, 'R^3', r'\mathbb{R}^3')
c_cart.<x,y,z> = R.chart()  # Cartesian coord. on R^3
phi = U.diff_map(R, (sin(th)*cos(ph), sin(th)*sin(ph), cos(th)),
				 name='phi', latex_name=r'\varphi'); print(phi)
\end{NBin}
\begin{NBout}
Differentiable map phi from the Open subset U of the 2-dimensional differentiable manifold S^2 to the 3-dimensional differentiable manifold R^3
\end{NBout}
Let us fix the point $p$ in $U \subset \Sphere^2$ determined by $(\frac{\pi}{2},\pi)$ in spherical coordinates:
\begin{NBin}
p = U.point((pi/2, pi), name='p'); print(p)	
\end{NBin}
\begin{NBout}
Point p on the 2-dimensional differentiable manifold S^2
\end{NBout}
We can evaluate $\varphi$ at this point $p$: 
\begin{NBin}
phi(p).coord(c_cart)
\end{NBin}
\begin{NBoutM}
$\left(-1, 0, 0\right)$
\end{NBoutM}
We get the corresponding pullback tensor bundles by stating $\varphi$ as the destination map:
\begin{NBin}
phiT11U = U.tensor_bundle(1,1, dest_map=phi); phiT11U
\end{NBin}
\begin{NBoutM}
$\varphi^* T^{(1,1)}\mathbb{R}^3\to U$
\end{NBoutM}
More precisely:
\begin{NBin}
print(phiT11U)
\end{NBin}
\begin{NBout}
Tensor bundle phi^*T^(1,1)R^3 over the Open subset U of the 2-dimensional differentiable manifold S^2 along the Differentiable map phi from the Open subset U of the 2-dimensional differentiable manifold S^2 to the 3-dimensional differentiable manifold R^3	
\end{NBout}
We see that sections completely fall back on the preexisting implementation of tensor fields:
\begin{NBin}
phiT11U.section_module() is U.tensor_field_module((1,1), dest_map=phi)
\end{NBin}
\begin{NBoutM}
$\mathrm{True}$
\end{NBoutM}
The fiber at $p$ is given by the space of $(1,1)$-tensors of the tangent space over $\RR^3$ at~$\varphi(p)$:
\begin{NBin}
phiT11U.fiber(p)
\end{NBin}
\begin{NBoutM}
$T^{(1, 1)}\left(T_{\varphi\left(p\right)}\,\mathbb{R}^3\right)$
\end{NBoutM}
Since $\RR^3$ is parallelizable, the pullback tensor bundle $\varphi^* T^{(1,1)}\mathbb{R}^3\to U$ must be trivial:
\begin{NBin}
phiT11U.is_manifestly_trivial()
\end{NBin}
\begin{NBoutM}
$\mathrm{True}$
\end{NBoutM}
Hence, it comes with a frame naturally induced by the pullback:
\begin{NBin}
phiT11U.frames()
\end{NBin}
\begin{NBoutM}
$\left[\left(U, \left(\frac{\partial}{\partial x },\frac{\partial}{\partial y },\frac{\partial}{\partial z }\right)\right)\right]$
\end{NBoutM}
Strictly speaking, this is a frame of $\varphi^* T\mathbb{R}^3$ rather than $\varphi^* T^{(1,1)}\mathbb{R}^3$. However, remember that all frames in the tensor bundle can be retrieved from frames in the tangent bundle. Thus, there is no loss of generality here. We can extract our frame by applying index operations on the returned list:
\begin{NBin}
print(phiT11U.frames()[0])
\end{NBin}
\begin{NBout}
Vector frame (U, (d/dx,d/dy,d/dz)) with values on the 3-dimensional differentiable manifold R^3
\end{NBout}

\section{Mixed Differential Forms}\label{ch:mixed}

\subsection{Mathematical Preliminaries}
Differential forms are a powerful tool in both geometry and physics. Intuitively, they correspond to sections of infinitesimal areas at each point on a manifold. We briefly recall some definitions and relations, especially with regards to characteristic classes. We begin with a definition.
\begin{mydef}\label{def:mixed:valued_forms}
	Suppose that $M$ and $N$ are smooth manifolds with $\dim(N)=n$. Let $E \to M$ be a smooth vector bundle over the field $\KK = \RR,\CC$. Moreover, assume that $\varphi : M \to N$ is a smooth map. For an open subset $U \subset M$ and $k \in \NN$, we define the following space of $C^\infty$-sections
	\begin{align*}
		\Omega^k(U, \varphi \,; E) := C^\infty\!\left( U \;;\; E \otimes_{\RR} {\bigwedge}^k ( \varphi^* T^{\,*}N ) \right),
	\end{align*}
	and call its elements \defstyle{$\boldsymbol k$-forms on $\boldsymbol U$ along $\boldsymbol\varphi$ with values in $\boldsymbol E$}.\footnote{Note that the additional vector bundle $E$ is a suitable generalization to allow more general coefficients. It is convenient, for example, with respect to \hyperref[par:vec:bundle_connections]{bundle connections}, see~Def.~\ref{def:vec:bundle_connection}.}
\end{mydef}
This space is an infinite dimensional vector space over $\mathbb{K}$ and a module over $C^\infty(U,\mathbb{K})$. As soon as $k$ exceeds $n$, we obviously obtain zero. The standard case of differential forms on $M$ is dedicated to $U=M=N$ with $\varphi=\id_M$ and the trivial bundle $E=\RR \times M$. Similarly, \emph{complex} differential forms on $M$ are given by $E=\CC \times M$ instead. For the sake of convenience, we agree on the following abbreviations:
\begin{align*}
	\Omega^k(U, \varphi \,; \KK) &:= \Omega^k(U, \varphi \,; \KK \times M) \\
	\Omega^k(U; E) &:= \Omega^k(U, \id_M \,; E).
\end{align*}

\paragraph{The Algebra of Mixed Forms.}
In this paragraph and what follows, we reduce $E$ to the trivial $\KK$-line bundle. With regards to characteristic classes, we need to perform algebraic operations on forms. Therefore, we define the space of \defstyle{mixed differential forms on $\boldsymbol U$ along $\boldsymbol\varphi$} by taking the following direct sum:
\begin{align*}
	\Omega^*(U, \varphi\,;\KK):=\bigoplus^n_{k=0} \Omega^k(U, \varphi \,; \KK).
\end{align*}
Let us make some observations. On the one hand, the space $\Omega^*(U, \varphi\,;\KK)$ is naturally endowed with an associative multiplication induced by the wedge product via bilinear extension:
\begin{align*}
	\wedge : \Omega^*(U, \varphi\,;\KK) \times \Omega^*(U, \varphi\,;\KK) \to \Omega^*(U, \varphi\,;\KK).
\end{align*}
This multiplication gives $\Omega^*(U, \varphi)$ the structure of a \emph{graded algebra}, i.\,e.
\begin{align*}
	\Omega^k(U, \varphi\,;\KK) \wedge \Omega^l(U, \varphi\,;\KK) \subset \Omega^{k+l}(U, \varphi\,;\KK).
\end{align*}
Observe that $\wedge$ is in general neither commutative nor anticommutative. Nonetheless, if one restricts to differential forms of even degree, the multiplication becomes commutative. Similarly, it gets anticommutative when one considers odd degrees.

On the other hand, the space $\Omega^*(U, \varphi\,;\KK)$ is equipped with an additional structure. More precisely, the \emph{exterior derivative} as a linear mapping
\begin{align*}
	\mathrm{d}_{k+1} : \Omega^k(U, \varphi\,;\KK) \to \Omega^{k+1}(U, \varphi\,;\KK)
\end{align*}
satisfies $\mathrm{d}_{k+1} \circ \mathrm{d}_k = 0$ and hence delivers the structure of a \emph{cochain complex}, namely
\begin{align*}
0 \xrightarrow{\diffd_0} \Omega^0(U, \varphi\,;\KK) \xrightarrow{\diffd_1} \Omega^1(U, \varphi\,;\KK) \xrightarrow{\diffd_2} \dots \xrightarrow{\diffd_n} \Omega^n(U, \varphi\,;\KK)  \xrightarrow{\diffd_{n+1}} 0.
\end{align*}
This naturally induces a $\KK$-linear map on the entire space:
\begin{align*}
\mathrm{d} : \Omega^*(U, \varphi\,;\KK) \to \Omega^*(U, \varphi\,;\KK).
\end{align*}
Differential forms in $\im(\diffd)$ are called \defstyle{exact}, while forms in $\ker(\diffd)$ are called~\defstyle{closed}.

\paragraph{De Rham Cohomology.}
In the course of this paragraph, we restrict our previous definitions to the smooth case $U=M=N$ with $\varphi = \id_M$; but still with a possible complexification in mind. The famous Poincaré lemma states that exact and closed differential forms always coincide on domains diffeomorphic to an open ball. Conversely, counterexamples indicate that the reason for failure are \enquote{holes} in the domain. This observation gives rise to geometric invariants on smooth manifolds. More precisely, we take closed forms and quotient out exact forms. The resulting space
\begin{align*}
	H_{\mathrm{dR}}^k(M;\KK) := \faktor{\ker(\diffd_{k+1})}{\im(\diffd_{k})}
\end{align*}
is called \defstyle{$\boldsymbol k$-th de Rham cohomology of $\boldsymbol M$ with coefficients in $\boldsymbol\KK$}. On top of that, we conveniently denote:
\begin{align*}
	H_{\mathrm{dR}}^*(M;\KK) := \bigoplus^n_{k=0} H_{\mathrm{dR}}^k(M;\KK).
\end{align*}
The space $H_{\mathrm{dR}}^*(M;\KK)$ inherits the structure of a graded algebra and cochain complex from $\Omega^*(M ;\KK)$ by construction. Moreover, the sequence $$\ldots \xrightarrow{\diffd_{k}} H_{\mathrm{dR}}^k(M;\KK) \xrightarrow{\diffd_{k+1}} H_{\mathrm{dR}}^{k+1}(M;\KK) \xrightarrow{d_{k+2}} \ldots$$ is \emph{exact}, i.\,e.\ $\ker(d_{k+1}) = \im(d_{k})$. One can similarly define more general cohomologies with coefficients in generic rings. For a detailed introduction into algebraic topology, and homology theory in particular, consult \cite{baer:algtop, hatcher:algebraic_topology}.

The de Rham cohomology yields a geometric invariant in the following sense: if $\psi : M \to N$ is a smooth map into another manifold $N$, the pullback $$\psi^* : \Omega^k(N;\KK) \to \Omega^k(M;\KK)$$ descends to a linear mapping $H_{\mathrm{dR}}^k(N;\KK) \to H_{\mathrm{dR}}^k(M;\KK)$, which is also denoted by $\psi^*$. Suppose $\psi$ is a diffeomorphism then $\psi^*$ boils down to an isomorphism on cohomology level. A full discussion on that topic can be found in \cite[Ch.~17]{lee:smooth_manifolds}.

\subsection{SageMath Implementation}\label{sec:mixed:impl}
Differential forms on manifolds came along with the \texttt{SageManifolds} project and hence are already supported in \Sage.\footnote{There, $\KK$ always equals the base field of the manifold, particularly $\KK = \RR$ for real manifolds. However, since symbolic expressions are used for local coordinates, there is no genuine restriction to~$\RR$.} A mixed differential form can simply be represented by an element-typed object storing differential forms of different degrees and provided with the structures discussed in the preceding section. More precisely, a mixed form is represented by an instance of \texttt{MixedForm} inheriting from \texttt{AlgebraElement}. The overlying algebra is an instance of the parent typed class \texttt{MixedFormAlgebra} which inherits from \texttt{Parent} and \texttt{UniqueRepresentation}, and belongs to the category \texttt{GradedAlgebras} over the symbolic ring. It is uniquely determined by the vector field module, which in turn is specified by $U$ and the destination map $\restrict{\varphi}{U}: U \to N$.

As the inheritances suggest, mixed forms are fully integrated in the parent-element-type pattern of \Sage. %
For example, the most canonical coercions are supported:
\begin{align*}
	\Omega^k(U, \varphi\,;\KK) &\xrightarrow[\phantom{\mathrm{restriction}}]{\mathrm{inclusion}} \Omega^*(U, \varphi\,;\KK), \\
	\Omega^*(M, \varphi\,;\KK) &\xrightarrow{\mathrm{restriction}} \Omega^*(U, \varphi\,;\KK).
\end{align*}
To ensure full compatibility, the class \texttt{MixedForm} is provided with all important methods which are already specified for differential forms. This includes, for example, \texttt{restrict}, \texttt{add\_comp\_by\_continuation} and \texttt{set\_restriction}. We refer to the \Sage reference manual~\cite{sagemath:reference:manif} for a complete list.

\begin{figure}[t]
	\centering
	\tikzstyle{arrow} = [thick,->,>=stealth]
	\tikzstyle{factory} = [rectangle, rounded corners, minimum width=2.2cm, minimum height=1cm,text centered, draw=black, fill=strcolor!30, thin]
	\tikzstyle{obj} = [rectangle, rounded corners, minimum width=2.2cm, minimum height=1cm,text centered, draw=black, fill=gray!50, thin]
	\pgfdeclarelayer{background}
	\pgfdeclarelayer{foreground}
	\pgfsetlayers{background,main,foreground}
	\begin{tikzpicture}[node distance=3.25cm, >=latex', thick]
	\begin{pgfonlayer}{foreground}
	\node (scal-py) [obj, align=center] {\scriptsize Scalar field: \\ \scriptsize $f: U \to \KK$};
	\node (oneform-py) [obj, right of=scal-py, align=center] {\scriptsize 1-form: \\ \scriptsize $\omega_1 \in \Omega^1(U,\varphi\,;\KK)$};
	\node (other-py) [right of=oneform-py, align=center,xshift=-.8cm] {\large \ldots};
	\node (nform-py) [obj, right of=other-py, align=center,xshift=-.8cm] {\scriptsize $n$-form: \\ \scriptsize $\omega_n \in \Omega^n(U,\varphi\,;\KK)$};
	
	\node (scal) [obj, below of=scal-py, align=center,yshift=.75cm] {$f$};
	\node (oneform) [obj, below of=oneform-py, align=center,yshift=.75cm] {$\omega_1$};
	\node (other) [below of=other-py, align=center,yshift=.75cm] {\large \ldots};
	\node (nform) [obj, below of=nform-py, align=center,yshift=.75cm] {$\omega_n$};
	
	\node (A0) [above of=scal-py,align=center,yshift=-2.4cm] {\texttt{A[0]}};
	\node (A1) [above of=oneform-py,align=center,yshift=-2.4cm] {\texttt{A[1]}};
	\node (An) [above of=nform-py,align=center,yshift=-2.4cm] {\texttt{A[n]}};

	\node at ($(scal-py.east)!0.5!(oneform-py.west)$) [align=center,yshift=-.25cm] {\large ,};
	\node at ($(oneform-py.east)!0.5!(other-py.west)$) [align=center,yshift=-.25cm] {\large ,};
	\node at ($(other-py.east)!0.5!(nform-py.west)$) [align=center,yshift=-.25cm] {\large ,};
	
	\draw [arrow,opacity=.8] (scal) -- (scal-py) node[right,midway,yshift=-.47cm] {coerce};
	\draw [arrow,opacity=.8] (oneform) -- (oneform-py) node[right,midway,yshift=-.47cm] {coerce};
	\draw [arrow,opacity=.8] (nform) -- (nform-py) node[right,midway,yshift=-.47cm] {coerce};
	\end{pgfonlayer}
	\begin{pgfonlayer}{main}
	\tikzset{list/.style={draw,dashed,gray,rounded corners,fill=kwcolor!50!gray!50!,inner sep=10pt}}
	\node[list,fit=(scal-py) (nform-py) (An), yshift=-.1cm] (list) {};
	\end{pgfonlayer}
	\begin{pgfonlayer}{foreground}
	\node (mixed-label)[above of=list,align=center,yshift=-1.4cm] {\bfseries Mixed form $\boldsymbol{A}$ in $\boldsymbol{\Omega^*(U, \varphi\,;\KK)}$:};
	\end{pgfonlayer}
	\begin{pgfonlayer}{background}
	\tikzset{mixed/.style={draw,black,rounded corners,fill=strcolor!30,inner sep=12pt}}
	\node[mixed,fit=(list) (mixed-label) (list)] {};
	\end{pgfonlayer}
	
	\end{tikzpicture}
	\caption{A mixed form $A$ with homogeneous components $(f, \omega_1, \ldots, \omega_n)$ and its representation \texttt{A} in \Sage. Before $f$ is assigned to the entry \texttt{A[0]}, $\omega_1$ to \texttt{A[1]} and so forth, each homogeneous component is coerced into a proper differential form.}\label{fig:mixed:mixed_form}
\end{figure}
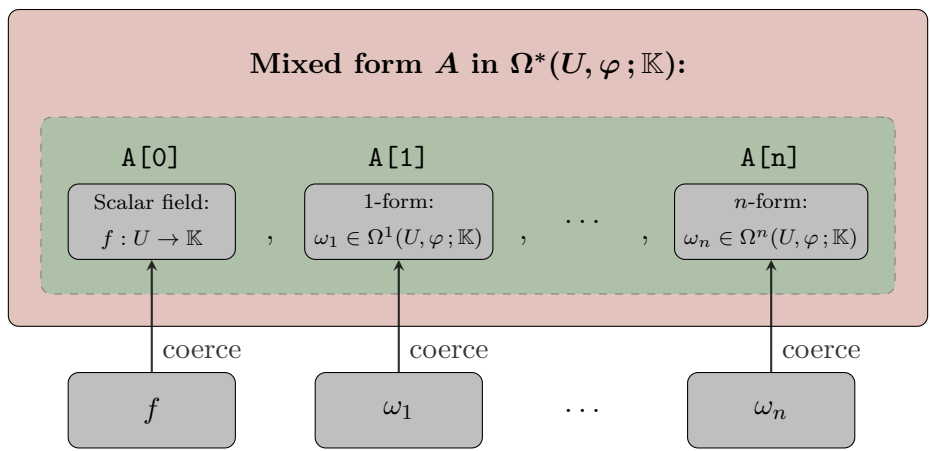

In order to realize mixed forms living in $\Omega^*(U, \varphi\,;\KK)$ appropriately, a \Python list of length $n+1$ is attached to each instance of \texttt{MixedForm}, containing differential forms of different degree in each entry. The entries can be accessed by using index operations on the object itself. Thereby, the $k$-th index corresponds to the $k$-th homogeneous component. Before a differential form is assigned this way, it gets coerced into a suitable \Sage element. This is usually an instance of \texttt{DiffForm}, or \texttt{DiffFormParal} if $U$ is parallelizable. In case of degree zero, however, a generic $0$-form is represented by an instance of \texttt{DiffScalarField}. An illustrated example is provided in Figure~\ref{fig:mixed:mixed_form}. Mathematical operations such as additions, multiplications and exterior derivatives are then performed degree wise.

\subsection{Examples}\label{sec:mixed:examples}

In this section, we briefly demonstrate the capabilities of our implementation. We start by declaring the manifold $M=\RR^2$:
\begin{NBin}
M = Manifold(2, 'R^2', latex_name=r'\mathbb{R}^2')
X.<x,y> = M.chart()
\end{NBin}
Next we define the corresponding spaces of differential forms:
\begin{NBin}
Omega0 = M.diff_form_module(0); print(Omega0)
Omega1 = M.diff_form_module(1); print(Omega1)
Omega2 = M.diff_form_module(2); print(Omega2)
\end{NBin}
\begin{NBout}
Algebra of differentiable scalar fields on the 2-dimensional differentiable manifold R^2
Free module Omega^1(R^2) of 1-forms on the 2-dimensional differentiable manifold R^2
Free module Omega^2(R^2) of 2-forms on the 2-dimensional differentiable manifold R^2
\end{NBout}
The algebra of mixed forms is returned by a simple command:
\begin{NBin}
Omega = M.mixed_form_algebra(); print(Omega)
\end{NBin}
\begin{NBout}
Graded algebra Omega^*(R^2) of mixed differential forms on the 2-dimensional differentiable manifold R^2
\end{NBout}
As aforementioned in the previous section, it belongs to the category of graded algebras over the symbolic ring:
\begin{NBin}
print(Omega.category())
\end{NBin}
\begin{NBout}
Category of graded algebras over Symbolic Ring
\end{NBout}
Before proceeding with mixed forms, let us first declare some differential forms:
\begin{NBin}
f = M.scalar_field(name='f')
omega1 = M.diff_form(1, name='omega_1', latex_name=r'\omega_1')
omega2 = M.diff_form(2, name='omega_2', latex_name=r'\omega_2')
eta = M.diff_form(1, name='eta', latex_name=r'\eta')
\end{NBin}
In the next step, we provide some expressions in local coordinates:
\begin{NBin}
f.set_expr(x^2)
omega1[:] = y, 2*x
omega2[0,1] = 4*x^3
eta[:] = x, y
\end{NBin}
\begin{NBin}
f.display()
\end{NBin}
\begin{NBoutM}
$\begin{array}{llcl} f:& \mathbb{R}^2 & \longrightarrow & \mathbb{R} \\ & \left(x, y\right) & \longmapsto & x^{2} \end{array}$
\end{NBoutM}
\begin{NBin}
omega1.display()
\end{NBin}
\begin{NBoutM}
$\omega_1 = y \mathrm{d} x + 2 \, x \mathrm{d} y$
\end{NBoutM}
\begin{NBin}
omega2.display()
\end{NBin}
\begin{NBoutM}
$\omega_2 = 4 \, x^{3} \mathrm{d} x\wedge \mathrm{d} y$
\end{NBoutM}
\begin{NBin}
eta.display()
\end{NBin}
\begin{NBoutM}
$\eta = x \mathrm{d} x + y \mathrm{d} y$
\end{NBoutM}
The category framework of \Sage captures the entire setup:
\begin{NBin}
all([f in Omega0,
	 f in Omega,
	 omega1 in Omega1,
	 omega1 in Omega,
	 omega2 in Omega2, 
	 omega2 in Omega])
\end{NBin}
\begin{NBoutM}
$\mathrm{True}$
\end{NBoutM}
Now, let us define a mixed form:
\begin{NBin}
A = M.mixed_form(name='A'); print(A)
\end{NBin}
\begin{NBout}
Mixed differential form A on the 2-dimensional differentiable manifold R^2
\end{NBout}
It shall consist of the differential forms $f, \omega_1, \omega_2$. As mentioned above, the forms are assigned by using index operations:
\begin{NBin}
A[:] = [f,omega1,omega2]; A.display()
\end{NBin}
\begin{NBoutM}
$A = f + \omega_1 + \omega_2$
\end{NBoutM}
\begin{NBin}
A.display_expansion()
\end{NBin}
\begin{NBoutM}
$A = \left[ x^{2} \right]_0 + \left[ y \mathrm{d} x + 2 \, x \mathrm{d} y \right]_1 + \left[ 4 \, x^{3} \mathrm{d} x\wedge \mathrm{d} y \right]_2$
\end{NBoutM}
As we can see, the output is sorted by degree. Notice that the forms stored in \texttt{A} are given by the very same instances we declared beforehand:
\begin{NBin}
all([A[0] is f,
	 A[1] is omega1,
	 A[2] is omega2])
\end{NBin}
\begin{NBoutM}
$\mathrm{True}$
\end{NBoutM}
If that behavior is unwanted, a copy can be made that has the very same expressions in local coordinates but has stored entirely new instances:
\begin{NBin}
Aclone = A.copy()
any(Aclone[k] is A[k] for k in Omega.irange())
\end{NBin}
\begin{NBoutM}
$\mathrm{False}$
\end{NBoutM}
\begin{NBin}
all(Aclone[k] == A[k] for k in Omega.irange())
\end{NBin}
\begin{NBoutM}
$\mathrm{True}$
\end{NBoutM}
Let us perform some computations and define another mixed form:
\begin{NBin}
B = M.mixed_form([2,eta,0], name='B'); B.display_expansion()
\end{NBin}
\begin{NBoutM}
$B = \left[ 2 \right]_0 + \left[ x \mathrm{d} x + y \mathrm{d} y \right]_1 + \left[ 0 \right]_2$
\end{NBoutM}
The multiplication is executed degree wise:
\begin{NBin}
all((A * B)[k] == sum(A[j].wedge(B[k - j]) 
	 				  for j in range(k + 1)) 
	for k in Omega.irange())
\end{NBin}
\begin{NBoutM}
$\mathrm{True}$
\end{NBoutM}
\begin{NBin}
(A * B).display_expansion()
\end{NBin}
\begin{NBoutM}
$A\wedge B = \left[ 2 \, x^{2} \right]_0 + \left[ \left( x^{3} + 2 \, y \right) \mathrm{d} x + \left( x^{2} y + 4 \, x \right) \mathrm{d} y \right]_1 + \left[ \left( 8 \, x^{3} - 2 \, x^{2} + y^{2} \right) \mathrm{d} x\wedge \mathrm{d} y \right]_2$
\end{NBoutM}
This particular example is also convenient to demonstrate that the multiplication given by the wedge product is in general neither commutative nor anticommutative:
\begin{NBin}
(B * A).display_expansion()
\end{NBin}
\begin{NBoutM}
$B\wedge A = \left[ 2 \, x^{2} \right]_0 + \left[ \left( x^{3} + 2 \, y \right) \mathrm{d} x + \left( x^{2} y + 4 \, x \right) \mathrm{d} y \right]_1 + \left[ \left( 8 \, x^{3} + 2 \, x^{2} - y^{2} \right) \mathrm{d} x\wedge \mathrm{d} y \right]_2$
\end{NBoutM}
Finally, let us compute the exterior derivative:
\begin{NBin}
dA = A.exterior_derivative(); dA.display_expansion()
\end{NBin}
\begin{NBoutM}
$\mathrm{d}A = \left[ 0 \right]_0 + \left[ 2 \, x \mathrm{d} x \right]_1 + \left[ \mathrm{d} x\wedge \mathrm{d} y \right]_2$
\end{NBoutM}

\section{Characteristic Classes}\label{ch:char_class}

\subsection{Mathematical Preliminaries}\label{sec:char:math}
There are a few equivalent definitions of characteristic classes. One of them is the following, characterized by the so-called naturality condition.
\begin{mydef}\label{def:char_class}
	Let $R$ be a commutative ring, $M$ a topological manifold and $E$ a topological vector bundle over $M$. A \defstyle{characteristic class} $\kappa(E)$ of $E$ is an element in the cohomology ring $H^*(M; R)$ with coefficients in $R$ such that for every continuous map $f: N \to M$, where $N$ is another topological manifold, the \defstyle{naturality condition} holds:
	\begin{align*}
		f^* \kappa(E) = \kappa(f^* E).
	\end{align*}
\end{mydef}
Roughly speaking, a characteristic class measures the vector bundle's \enquote{non-triviality} in a certain way. Typically, there are various methods to obtain characteristic classes. In this thesis, we make use of the so-called \defstyle{Chern--Weil method}, which utilizes connections on differentiable vector bundles to construct characteristic classes in the de Rham cohomology ring with complex coefficients of the underlying differentiable manifold. This seems to be a good approach since the curvature, in a sense, measures the local deviation from flatness. We give a brief review of Chern--Weil theory in this chapter, inspired by \cite{roe:elliptic, baer:charakteristisch}. A full discussion of characteristic classes in a more general setting can be found in~\cite{milnor:char_classes}. In what follows, unless stated otherwise, we consider manifolds and vector bundles to be smooth.

\paragraph{Chern--Weil Theory.} Let us first recall some fundamental definitions and theorems about invariant polynomials and finally introduce the Chern--Weil theory. We start with a definition:
\begin{mydef}
	Let $G$ be a compact Lie group and $\mathfrak{g}$ its Lie algebra. A polynomial $P : \mathfrak{g} \to \CC$ is called \defstyle{invariant} if
	\begin{align*}
		P(X) = P(\mathrm{Ad}_g(X))
	\end{align*}
	holds for every $X \in \mathfrak{g}$ and $g \in G$.
\end{mydef}
A canonical choice in our setting would be $G = \mathrm{GL}(n,\CC)$ and $\mathfrak{g}=\mathfrak{gl}(n,\CC)$ with the adjoint representation $\mathrm{Ad}_g(X) = g Xg^{-1}$. Famous examples of invariant polynomials on $\mathfrak{gl}(n,\CC)$ are the determinant and trace. However, we can consider $G$ being any other Lie group; for example, the group $G=\mathrm{SO}(2n)$ with the Lie algebra $\mathfrak{g}=\mathfrak{so}(2n)$ of skew-symmetric matrices. This special case plays an important role for oriented vector bundles when we discuss Pfaffian classes later in this section.

The following lemma is essential with respect to characteristic classes.
\begin{mylem}\label{lem:elem_sym_pol_gen}
	The ring of invariant polynomials on $\mathfrak{gl}(n,\CC)$ is a polynomial ring generated by the elementary symmetric functions $\sigma_k$ specified by the characteristic polynomial of a complex matrix $X \in \mathfrak{gl}(n,\CC)$ in the indeterminate~$t$:
	\begin{align}
		\det\left(1 + \frac{t X}{2 \pi i}\right) = \sum_{k=0}^{n} \sigma_k(X) \, t^k . \label{eq:elem_sym_pol}
	\end{align}
\end{mylem}
See \cite[p. 299, Lemma~6]{milnor:char_classes}, \cite[Lem.~2.19]{roe:elliptic} for proofs. We leave this lemma behind for the moment and focus on the fundamental theorem of Chern--Weil theory.
\begin{mythm}
	Let $E$ be a complex vector bundle over a manifold $M$ equipped with a connection $\nabla$, and let $\Omega$ be its curvature form matrix with respect to some frame. Suppose $P : \mathfrak{gl}(n,\CC) \to \CC$ is an invariant polynomial. Then the form $P(\Omega)$ is closed and
	\begin{align*}
		\left[ P(\Omega) \right] \in H^{2*}_{\mathrm{dR}}(M; \CC)
	\end{align*}
	is independent of the choice of connection $\nabla$.
	\label{thm:inv_ind_con}
\end{mythm}
Proofs can be found in \cite[296~ff.]{milnor:char_classes}, \cite[Lem.~1.8, Satz~1.9]{baer:charakteristisch}, \cite[297~ff.]{eguchi:grav}, \cite[Prop.~2.20]{roe:elliptic}. Notice that the form $P(\Omega)$ is a well-defined global form on $M$. This follows immediately from the transformation behavior of $\Omega$ together with the invariant nature of $P$. We can even consider $P$ to be an invariant formal power series. Due to the fact that $\Omega$ is always nilpotent, no convergence questions arise and the same statement still applies. Finally, from Theorem~\ref{thm:inv_ind_con}, we see that $\left[ P(\Omega) \right]$ defines a characteristic class in the sense of Definition~\ref{def:char_class}. This boils down to the fact that pullback operations pass through curvatures and polynomials up to cohomology level.

These are all the ingredients we need for the Chern--Weil method. The recipe manifests as follows: \enquote{Choose a connection, compute its curvature form matrix and apply an invariant formal power series to it}. Before we discuss the implementation details in \Sage, let us introduce some important types of classes constructed from the Chern--Weil method and investigate its properties.

\paragraph{Multiplicative Classes on Complex Bundles.}\label{para:mul_chern_genera}
Let $f$ be a holomorphic function near zero and $E$ be a complex vector bundle of rank $n$ over a manifold $M$. Furthermore, let $\nabla$ be a connection on $E$ and $\Omega$ its curvature form matrix with respect to some frame. We set
\begin{align*}
	\Pi_f(E, \nabla) := \det\left( f\left( \frac{\Omega}{2 \pi i} \right) \right) \in \Omega^{2*}(M; \CC)
\end{align*}
and define the \defstyle{multiplicative characteristic class associated to $\boldsymbol f$}:
\begin{align*}
	 \Pi_f(E) := \left[\Pi_f(E, \nabla)\right] \in H^{2*}_{\mathrm{dR}}(M; \CC).
\end{align*}
The form $\Pi_f(E, \nabla)$ is called \defstyle{characteristic form with respect to $\boldsymbol\nabla$} of the corresponding characteristic class $\Pi_f(E)$.

The most important multiplicative class is the so-called \defstyle{total Chern class} $c(E)$. It is defined as the multiplicative class associated to the function $$f(z) = 1+z.$$ The total Chern class is an element of the cohomology ring with purely real coefficients. This can be seen as follows. We equip $E$ with a Hermitian metric and compatible connection. Thus by choosing an orthonormal frame, the associated curvature form matrix is skew-Hermitian. Finally applying complex conjugation yields the assertion.
Comparison with~\eqref{eq:elem_sym_pol} reveals that the total Chern class breaks up as a sum of classes given by the elementary symmetric functions. The class $c_k(E)$ associated to the $k$-th elementary symmetric function $\sigma_k$ is called \defstyle{$\boldsymbol k$-th Chern class} of $E$. Its importance is seen by Lemma~\ref{lem:elem_sym_pol_gen}, which shows that every characteristic class constructed via the Chern--Weil method can be obtained from polynomials in the Chern classes $c_k(E)$.

Note the following two significant properties of multiplicative classes:
\begin{myenum}
	\item If $E$ is a complex line bundle, we have $\Pi_f(E) = f(c_1(E))$.
	\item If $E = E_1 \oplus E_2$, we have $\Pi_f(E) = \Pi_f(E_1)\, \Pi_f(E_2)$.
\end{myenum}
Property ii) can be easily seen by picking a direct sum connection and finally justifies the name \enquote{multiplicative}. In fact, both properties determine the characteristic class $\Pi_f(E)$ uniquely. This follows from the \defstyle{splitting principle}.
\begin{mythm}[Splitting Principle]
	Let $E$ be a complex vector bundle of rank $n$ over a manifold $M$ and $R$ be a commutative ring. Then, there exist another manifold $F(E)$, called \defstyle{flag manifold} associated to $E$, and a map $g: F(E) \to M$ such that
	\begin{myenum}
		\item the induced cohomology homomorphism $$g^* : H^*(M;R) \to H^*(F(E);R)$$ is injective and
		\item the pullback bundle $g^*E$ splits into the direct sum of line bundles: $$g^*E = L_1 \oplus \dots \oplus L_n.$$
	\end{myenum}
\end{mythm}
For the proof, we refer to~\cite[Prop.~11.1]{lawson:spin},~\cite[Thm.~2.6.1]{randal:characteristic}. A full discussion with respect to $K$-theory can be found in~\cite{hatcher:vector_bundles}. As usual, we restrict ourselves to the case $R=\RR, \CC$. The first Chern classes $c_1(L_j)$ of the corresponding line bundles $L_j$ into which $g^*E$ splits are called \defstyle{Chern roots} of $E$. We find that they determine the $k$-th Chern class of $E$ in the following way:
\begin{align*}
	c_k(E) = \sigma_k \left( c_1(L_1), \ldots, c_1(L_n) \right).
\end{align*}
Now, suppose $(x_j)^n_{j=1}$ are the eigenvalues including multiple appearances of a complex matrix ${X}/{2 \pi i}$. Then we obtain
\begin{align}
	\det\left(f\left(\frac{X}{2 \pi i}\right)\right) = \prod_{j=1}^{n} f(x_j), \label{eq:chern_to_eigenvalues}
\end{align}
which is a symmetric formal power series in the variables $x_j$. Hence, we can express it in terms of the elementary symmetric functions in these indeterminates $x_j$. This means, each $x_j$ can be considered as representing the Chern root $c_1(L_j)$. This result again emphasizes the fundamental significance of Chern classes.

The theory needs to be handled with greater precaution in the real case. Let $E$ be a real vector bundle and $E \otimes_{\RR} \CC$ its complexification. Suppose $X$ is a real skew-symmetric matrix. Due to~\eqref{eq:elem_sym_pol}, we see that
\begin{align*}
	c_k(X) = c_k(X^t) = c_k(-X) = (-1)^k c_k(X)
\end{align*}
from which we immediately infer $c_{2k+1}(X) = 0$. Since every real vector bundle carries a metric and hence a compatible connection, the curvature form matrix can be chosen skew-symmetric. In conclusion, the odd Chern classes of $E \otimes_{\RR} \CC$ vanish. It is convenient to denote
\begin{align*}
	p_k := (-1)^k c_{2k}
\end{align*}
from which we gain a new class $p_k(E) \in H^{4k}(M;\RR)$ called \defstyle{$\boldsymbol k$-th Pontryagin class} of $E$. The corresponding sum $$p(E) = p_0(E) + p_1(E) + \ldots \in H^{4*}_{\mathrm{dR}}(M;\RR)$$ is called \defstyle{total Pontryagin class} of $E$. This leads to an analogous theory of multiplicative classes for real vector bundles.

\paragraph{Multiplicative Classes on Real Bundles.}
Suppose $E$ is a real vector bundle of rank $n$ and $g$ is a holomorphic function near zero with $g(0)=1$. We define
\begin{align}
	f(z) := \sqrt{g(z^2)}
	\label{eq:f_from_g_genus}
\end{align}
to be the branch satisfying $f(0)=1$. We call the multiplicative characteristic class $\Pi_f(E \otimes_{\RR} \CC)$ the \defstyle{multiplicative class associated to $\boldsymbol g$} on the \emph{real} vector bundle $E$.

Remember \eqref{eq:chern_to_eigenvalues} and its role for the splitting principle. We have a quite similar equation in terms of Pontryagin classes, explained by the following lemma.
\begin{mylem}
	The multiplicative class on a real vector bundle $E$ associated to a function $g$ as above is equal to
	\begin{align*}
		\prod^{\lfloor n/2 \rfloor}_{j=1} g(y_j)
	\end{align*}
	where the Pontryagin classes of $E$ are the elementary symmetric functions in the indeterminates $y_j$.
\end{mylem}
The proof can be found in~\cite[Lem.~2.27]{roe:elliptic}. This justifies the name \enquote{multiplicative} for real vector bundles. Moreover, this explains the choice of squares and roots in~\eqref{eq:f_from_g_genus}. Furthermore, this lemma can be interpreted in terms of a splitting principle: there is a pullback of $E$ such that it splits into a sum of suitable real 2-plane bundles (and additionally one line bundle when the rank is odd), where $y_j$ represents the first Pontryagin class of the $j$-th 2-plane bundle. For details see~\cite[Prop.~11.2, Obs.~11.8]{lawson:spin}.

\paragraph{Additive Classes on Complex Bundles.}
Let us introduce characteristic classes of additive type. For this, we assume $E$ to be a complex vector bundle over a manifold $M$ with connection $\nabla$ and curvature form matrix $\Omega$ again. We denote 
\begin{align*}
	\Sigma_f(E, \nabla) := \tr\left( f\left( \frac{\Omega}{2 \pi i} \right) \right) \in \Omega^{2*}(M; \CC)
\end{align*}
and define the \defstyle{additive characteristic class associated to $\boldsymbol f$} to be
\begin{align*}
	\Sigma_f(E) := \left[ \Sigma_f(E, \nabla) \right] \in H^{2*}_{\mathrm{dR}}(M; \CC).
\end{align*}
The term \enquote{additive} comes from the obvious property
\begin{align*}
	\Sigma_f(E_1 \oplus E_2) = \Sigma_f(E_1) + \Sigma_f(E_2).
\end{align*}

The most prominent additive class is the so-called \defstyle{Chern character} $\mathrm{ch}(E)$ of $E$. It lives in $H^{2*}_{\mathrm{dR}}(M; \RR)$ and is associated to the function $$f(z) = \exp\left( z \right).$$
The Chern character plays an important role since it defines some sort of \enquote{ring homomorphism} via
\begin{align*}
	\mathrm{ch}(E_1 \oplus E_2) &= \mathrm{ch}(E_1) + \mathrm{ch}(E_2), \\
	\mathrm{ch}(E_1 \otimes E_2) &= \mathrm{ch}(E_1)\, \mathrm{ch}(E_2).
\end{align*}
The second equation can be obtained from the well-known behavior of the exponential map $\expe^{z_1}\expe^{z_2}=\expe^{z_1+z_2}$ and using the tensor product connection.

\paragraph{Additive Classes on Real Bundles.}
As previously discussed in the multiplicative case, real vector bundles need a special treatment in the additive case, too. Again, let $E$ be a real vector bundle over a manifold $M$ and $E \otimes_{\RR} \CC$ its complexification. Furthermore, let $g$ be a holomorphic function near zero with $g(0)=0$. We denote
\begin{align*}
	f(z) := \frac{g(z^2)}{2}
\end{align*}
and define $\Sigma_f(E \otimes_{\RR} \CC)$ to be the \defstyle{additive characteristic class associated to $\boldsymbol g$} of the \emph{real} vector bundle $E$. The following lemma adapted from \cite[Lem.~2.27]{roe:elliptic} explains the factor $1/2$.
\begin{mylem}
	The additive characteristic class associated to a function $g$ as above of a real vector bundle $E$ is equal to
	\begin{align*}
		\sum^{\lfloor n/2 \rfloor}_{j=1} g(y_j)
	\end{align*}
	where the Pontryagin classes of $E$ are the elementary symmetric functions in the indeterminates $y_j$. 
\end{mylem}
\begin{proof}
	In the real case, the curvature matrix can be chosen skew-symmetric and is therefore similar to one in block diagonal form. The corresponding blocks are $2 \times 2$ and of the form
	\begin{align*}
		X = \begin{pmatrix}0 & \lambda \\ -\lambda & 0 \end{pmatrix}
	\end{align*}
	with eigenvalues $\pm i \lambda$ and possibly one $1 \times 1$ block equal to zero. Since both sides of the stated equality are additive for direct sums, it is sufficient to prove it for either block type. The zero case is trivial. For the case of block type $X$, we have
	\begin{align*}
		c_1(X) = 0, \quad c_2(X) = - \frac{\lambda^2}{4 \pi^2}
	\end{align*}
	and thus $y=p_1(X) = \frac{\lambda^2}{4 \pi^2}$. On the other hand, the matrix $X$ is similar to
	\begin{align*}
		\begin{pmatrix}- i \lambda & 0 \\ 0 & i \lambda \end{pmatrix}
	\end{align*}
	over the complex field, and
	\begin{align*}
		\tr\left( f({X}/{2\pi i}) \right) = f({\lambda}/{2 \pi}) + f(-{\lambda}/{2 \pi}) = 2 f({\lambda}/{2 \pi}) = g({\lambda^2}/{4 \pi^2}) = g(y)
	\end{align*}
	holds as required.
\end{proof}

\paragraph{Pfaffian Classes.}
On oriented real vector bundles with even rank, there is a characteristic class that cannot be directly expressed in terms of Pontryagin classes. Let $X \in \mathfrak{so}(2n)$ be a skew-symmetric matrix and $(e_1, \dots, e_{2n})$ be the standard basis of $\RR^{2n}$. We put
\begin{align*}
	\alpha = \sum_{i < j} X_{ij} \, e_i \wedge e_j
\end{align*}
and define the so-called \defstyle{Pfaffian} $\Pf(X) \in \RR$ via
\begin{align*}
	\frac{1}{n!} \alpha^n = \Pf(X) \; e_1 \wedge \dots \wedge e_{2n}.
\end{align*}
From this definition, we infer some essential properties stated by the following lemma.
\begin{mylem}\label{lem:pfaffian}
	Suppose $X\in \mathfrak{so}(2n)$ is a real skew-symmetric matrix. Then the following three statements hold:
	\begin{myenum}
		\item $\Pf(\lambda X) = \lambda^n \Pf(X)$ for $\lambda \in \RR$,
		\item $\Pf(B X B^t) = \det(B) \Pf(X)$ for $B \in \mathfrak{gl}(n, \RR)$,
		\item $\Pf(X)^2 = \det(X)$.
	\end{myenum}
\end{mylem}
\begin{proof}
	Claim i) is obvious and follows evidently from the definition. Now, assume $B \in \mathfrak{gl}(n, \RR)$ and set $f_i = \sum_{k} B_{ki}\, e_k$. Then we get:
	\begin{align*}
		n!\, \Pf(B X B^t) \, e_1 \wedge \dots \wedge e_{2n} &= \left( \frac{1}{2} \sum_{i,j} (B X B^t)_{ij} \, e_i \wedge e_j \right)^n \\
			&= \left( \frac{1}{2} \sum_{i,j} \sum_{k,l} B_{ik} X_{kl} B_{jl} \, e_i \wedge e_j \right)^n \\
			&= \left( \frac{1}{2} \sum_{k,l} X_{kl} \, f_k \wedge f_l \right)^n \\
			&= n!\, \Pf(X)\, f_1 \wedge \dots \wedge f_{2n} \\
			&= n!\, \Pf(X) \det(B) \, e_1 \wedge \dots \wedge e_{2n}.
	\end{align*}
	This yields the desired result of assertion ii). To deal with the remaining case iii), we make use of the fact that $X$ is similar~to
	\begin{align*}
		\mathrm{diag}\left( \begin{smallpmatrix} 0 & \lambda_1 \\ -\lambda_1 & 0 \end{smallpmatrix}, \dots, \begin{smallpmatrix} 0 & \lambda_n \\ -\lambda_n & 0 \end{smallpmatrix} \right)
	\end{align*}
	using an orthogonal transformation $T \in \mathrm{SO}(2n)$ which is orientation preserving. On the one hand, by using ii), we see
	\begin{align*}
		\Pf(X) \, e_1 \wedge \dots \wedge e_{2n} &= \Pf(TXT^t) \, e_1 \wedge \dots \wedge e_{2n} \\
			&= \frac{1}{n!} \left( \lambda_1 \, e_1 \wedge e_2 + \dots + \lambda_n \, e_{2n-1} \wedge e_{2n} \right)^n \\
			&= \left( \lambda_1 \cdot \ldots \cdot \lambda_n \right) \, e_1 \wedge \dots \wedge e_{2n}.
	\end{align*}
	On the other hand, we have
	\begin{align*}
		\det(X) = \det(TXT^t) = \lambda_1^2 \cdot \ldots \cdot \lambda_n^2.
	\end{align*}
	This confirms the required identity.
\end{proof}
Statement ii) reveals that the Pfaffian is an invariant polynomial on $\mathfrak{so}(2n)$ and gives rise to a characteristic class. Consider an oriented real vector bundle $E$ of rank $2n$ over a manifold $M$ and let $g$ be an odd, real analytic function near zero. We fix a metric on $E$ and choose a compatible connection $\nabla$. Restricted to oriented orthonormal frames on suitable subsets $U \subset M$, the corresponding curvature form matrix $\Omega$ is skew-symmetric; and so is $g(\Omega)$. If we change to another frame using $h: U \to \mathrm{SO}(2n)$, the curvature transforms as $h \Omega h^t$. Hence, we can define the global object
\begin{align*}
(\Pf)_g(E, \nabla) := \Pf\left( g\left( \frac{\Omega}{2 \pi} \right) \right) \in \Omega^{2n*}(M; \RR)
\end{align*}
and accordingly the \defstyle{Pfaffian class associated to $\boldsymbol g$}:
\begin{align*}
(\Pf)_g(E) := \left[\Pf_g(E, \nabla)\right] \in H_{\mathrm{dR}}^{2n*}(M; \RR).
\end{align*}
The proof of well-definedness is the same as before except in the use of invariant polynomials on $\mathfrak{so}(2n)$. Due to assertion i) of Lemma~\ref{lem:pfaffian}, we indeed see that the Pfaffian class lives in $H^{2n*}(M; \RR)$. Finally, the last property iii) is quite interesting in terms of the splitting principle. As the proof indicates, the Pfaffian class is equal to
\begin{align*}
	\prod_{j=1}^{n} g(\sqrt{y_j})
\end{align*}
where the Pontryagin classes are the elementary symmetric functions in the indeterminates $y_j$. So in a sense, the Pfaffian class describes a class of \enquote{square-root type}.

The most important Pfaffian class is induced by the function $g(x)=x$ and is called \defstyle{Euler class}, which is denoted by $e(E)$. Its prominent status is explained by the following theorem.
\begin{mythm}[Gauß--Bonnet--Chern]\label{thm:gauss_bonnet_chern}
	Let $M$ be a closed, oriented, $2n$-dimensional Riemannian manifold. Then we have
	\begin{align*}
		\chi(M) = \int_{M} e(TM),
	\end{align*}
	where $\chi(M)$ denotes the Euler characteristic of $M$ and $e(TM)$ the Euler class of the tangent bundle $TM$.
\end{mythm}
The original proof by Chern can be found in~\cite{chern:gauss_bonnet}. The \Sage example in Section~\ref{sec:char:euler} provides an application of this theorem for the 2-sphere.

\paragraph{Index Theorems.}
It would be an injustice to omit mentioning the connection between characteristic classes and indices of certain elliptic operators on closed manifolds via the Atiyah--Singer index theorem~\cite{atiyah:asii}. In the previous paragraph, we already saw a special case of this index theorem: Gauß--Bonnet--Chern, cf.\ Theorem~\ref{thm:gauss_bonnet_chern}. Another important subcase is devoted to the Dirac operator:
\begin{mythm}[Atiyah--Singer]
	Suppose $M$ is a closed, even-dimensional Riemannian spin manifold and $E$ a complex vector bundle over $M$. Then the index of the classical twisted Dirac operator~$D$ satisfies:
	\begin{align}
		\ind(D) = \int_{M} \mathrm{ch}(E) \hat{A}(TM).\label{eq:dirac_index}
	\end{align}
\end{mythm}
Here, $\hat{A}(TM)$ denotes the $\hat{A}$-class on the tangent bundle $TM$, cf. Table~\ref{tab:prominent_char}. In case of compact manifolds with boundary, the expression in~\eqref{eq:dirac_index} takes a slightly different form. In fact, the integral depends on the choice of the characteristic form and the Dirac operator in general fails to be Fredholm for an arbitrary boundary condition~\cite{atiyah-patodi-singer:aps001}. In the Lorentzian setting, Christian Bär and Alexander Strohmaier provide an index theorem for the Dirac operator on manifolds with compact spacelike Cauchy boundary~\cite{baer:math_chiral}. The formula still looks similar to~\eqref{eq:dirac_index}, but additional quantities are involved due to the boundary. This particular theorem plays a significant role in physical applications like chiral anomalies~\cite{baer:chiral}. In Section~\ref{sec:char:ch_char} and~\ref{sec:char:berger}, we present two examples within \Sage which are related to this theorem.
\begin{table}[t]
	\centering
	\def\arraystretch{1.4}
	\small
	\begin{tabular}{lcccl}
		\toprule
		\textbf{Class Name} & \textbf{Shortcut} & \textbf{Field} & \textbf{Type} & \textbf{Function} \\
		\midrule
		Chern & $c(E)$ & complex & multiplicative & $f(z) = 1+z$ \\ 
		Chern character & $\mathrm{ch}(E)$ & complex & additive & $f(z) = \exp(z)$ \\ 
		Todd & $\mathrm{Td}(E)$ & complex & multiplicative & $f(z) = \frac{z}{1-\expe^{-z}}$ \\ 
		Pontryagin & $p(E)$ & real & multiplicative & $g(z) = 1+z$ \\ 
		A-Hat & $\hat{A}(E)$ & real & multiplicative & $g(z) = \frac{\sqrt{z}/2}{\sinh(\sqrt{z}/2)}$ \\ 
		Hirzebruch & ${L}(E)$ & real & multiplicative & $g(z) = \frac{\sqrt{z}}{\tanh(\sqrt{z})}$ \\
		Euler & $e(E)$ & real & Pfaffian & $g(x) = x$ \\
		\bottomrule
	\end{tabular}
	\caption{Prominent characteristic classes on a vector bundle $E$.}\label{tab:prominent_char}
\end{table}

A whole bunch of other index formulas arises when using different kinds of characteristic classes. Although we omit details here, a list of the most prominent classes related to index theory can be found in Table~\ref{tab:prominent_char}.

\subsection{SageMath Implementation}
The idea of our implementation follows the subsequent maxim: a characteristic class in \Sage represents a \enquote{factory}, manufacturing its corresponding characteristic forms out of bundle connections.

Suppose $E$ is a vector bundle of rank $n$ over a manifold $M$, and consider a characteristic class $\kappa(E)$ of $E$, being one of the aforementioned types. Hence, we fix a suitable holomorphic function $g$ near zero, and let $P$ be the invariant polynomial on which the class is based. This is either the trace for \emph{additive classes}, the determinant for \emph{multiplicative classes} or the Pfaffian for \emph{Pfaffian classes}. Within \Sage, $\kappa(E)$ is represented by an instance of \texttt{CharacteristicClass}. This \Python class inherits from \texttt{UniqueRepresentation} and \texttt{SageObject}. It is uniquely determined by the vector bundle $E \to M$, the given class type resolving $P$ and a symbolic expression $g(x)$ reflecting the holomorphic function, but also by its name. The initialization is illustrated through the vertical axis in Figure~\ref{fig:char:general_workflow}. Thereby, we integrate a bunch of commonly used characteristic classes in \Sage, still following this construction procedure in the background. Table~\ref{tab:prominent_char} contains a list of all accessible characteristic classes implemented at this stage.

\begin{figure}[t]
	\centering
	\tikzstyle{arrow} = [thick,->,>=stealth]
	\tikzstyle{factory} = [rectangle, rounded corners, minimum width=3cm, minimum height=1cm,text centered, draw=black, fill=strcolor!50, thin]
	\tikzstyle{obj} = [rectangle, rounded corners, minimum width=3cm, minimum height=1cm,text centered, draw=black, fill=gray!30, thin]
	\begin{tikzpicture}[node distance=\textwidth / 3, >=latex', thick]
	\node (char-class) [factory, align=center] {{\scriptsize Characteristic class:} \\ $\kappa(E)$};
	\node (bundle) [obj, above of=char-class, align=center,yshift=-2.25cm] {{\scriptsize Vector Bundle:} \\ $E \to M$};
	\node (con) [obj, left of=char-class, align=center] {{\scriptsize Bundle connection:} \\ $\nabla$};
	\node (char-form) [obj, right of=char-class, align=center] {{\scriptsize Characteristic form:} \\ $\kappa(E,\nabla)$};
	\node (sym-expr) [obj, below of=char-class, align=center,yshift=.9cm] {{\scriptsize Symbolic expression:} \\ $g(x)$};
	\node (sign) [obj, below of=char-class, align=center,yshift=2.25cm] {{\scriptsize Class type:} \\ {\tiny \texttt{'multiplicative'} , \dots}};
	\node (dict) [obj, dashed, below of=con, align=center,opacity=.7,yshift=2.5cm] {{\scriptsize Dict. of curv. matr.:} \\ \scriptsize$\left\{ e_1 : \Omega_1 , \ldots , e_m : \Omega_m \right\}$};
	
	\draw [arrow,kwcolor] (bundle) -- (char-class) node[midway,right,kwcolor] {\small \itshape init};
	\draw [arrow,dashed,opacity=.7] (con) -- (dict);
	\draw [arrow,ccolor,dashed,opacity=.7] (dict) -- (char-class.west) node[midway,below,sloped,ccolor,opacity=.7] {\small \itshape in};
	\draw [arrow,ccolor] (char-class) -- (char-form) node[midway,above,sloped,ccolor] {\small \itshape out};
	\draw [arrow,kwcolor] (bundle.west) -- (con.north) node[midway,above,sloped,kwcolor] {\small \itshape init};
	\draw [arrow,ccolor] (con) -- (char-class) node[midway,above,ccolor,sloped] {\small \itshape in};
	\draw [arrow,kwcolor] (sign) -- (char-class) node[midway,right,kwcolor] {\small \itshape init};
	\draw [kwcolor] (sym-expr) -- (sign);
	\end{tikzpicture}
	\caption{In the vertical axis we see the initialization dependencies, while the horizontal axis illustrates the \enquote{factory} behavior of a characteristic class in \Sage.}\label{fig:char:general_workflow}
\end{figure}
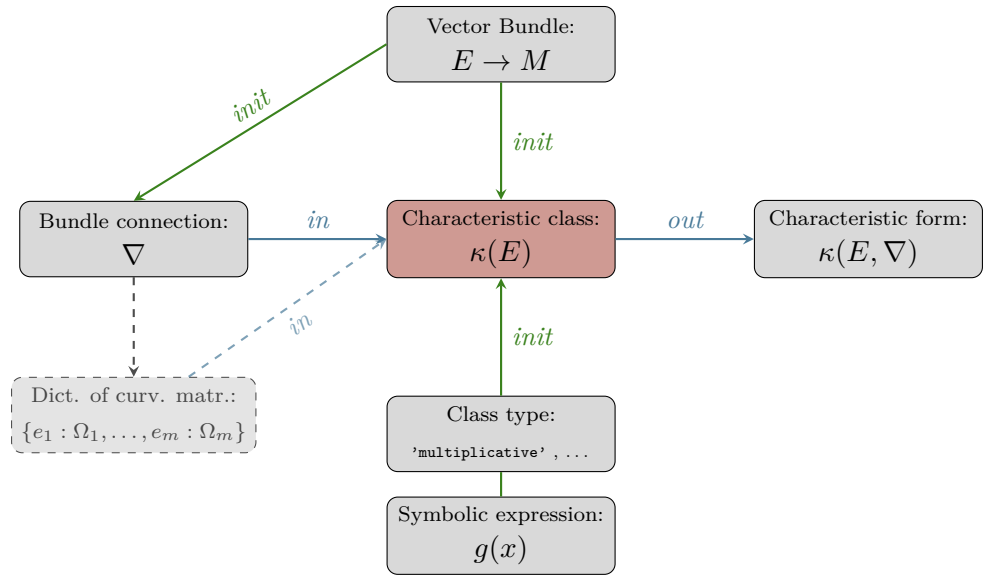

We use the implemented method \texttt{get\_form} to initiate the building process by taking a bundle connection $\nabla$ on $E$. Usually, our \Sage implementation tries to compute its curvature forms automatically and puts them into a \Python\ dictionary to eventually assemble the characteristic form. The dictionary is of the form ${\{ e_1 : \Omega_1, \ldots, e_m : \Omega_m \}}$ and consists of curvature form matrices with their associated local frames as keys.\footnote{The choice using local frames as keys is principally motivated from mathematics. This has the advantage that it might be useful for internal checks or further computations in the future.} This automatic computation, however, is unwanted at some point. For instance, assume the curvature is barely computable but already known. Or suppose the class is of Pfaffian type where the curvature form matrices must be considered skew-symmetric at this stage. For that purpose, our implementation allows an optional input of such dictionaries by hand. Though, in order to allocate the characteristic form appropriately, stating a connection is still compulsory. The general workflow is illustrated on the horizontal axis of Figure~\ref{fig:char:general_workflow}. Notice that the result is stored because long computation times are usually expected.

\paragraph{Detailed Workflow.}
We want to take a further look at the workflow and unravel it in detail. Figure~\ref{fig:char:specific_workflow} provides a sketch of how characteristic forms are constructed in our \Sage implementation. As we can see, the diagram splits in two main parts. 

\begin{figure}[t]
	\centering
	\tikzstyle{arrow} = [thick,->,>=stealth]
	\tikzstyle{factory} = [rectangle, rounded corners, minimum width=3cm, minimum height=1cm,text centered, draw=black, fill=strcolor!50, thin]
	\tikzstyle{obj} = [rectangle, rounded corners, minimum width=3cm, minimum height=1cm,text centered, draw=black, fill=gray!30, thin]
	\pgfdeclarelayer{background}
	\pgfdeclarelayer{foreground}
	\pgfsetlayers{background,main,foreground}
	\begin{tikzpicture}[node distance=\textwidth / 8, >=latex', thick]
	\node (sym-expr) [obj, align=center] {{\scriptsize Symbolic expression:} \\ $\displaystyle g(x)$};
	\node (sym-expr-mod) [obj, below of=sym-expr, align=center] {{\scriptsize Symbolic expression:} \\ $\displaystyle f(x)$};
	\node (taylor) [obj, below of=sym-expr-mod, align=center] {{\scriptsize List of coeff.:} \\ $\left[c_0, \ldots, c_k \right]$};
	\node (insert) [obj, below of=taylor, align=center,yshift=-.9cm] {{\scriptsize $\Omega^{2*}\!(U_\ell)$-valued matrix:} \\ $\scriptstyle f\left(\frac{\Omega_\ell}{2 \pi \varepsilon}\right) = c_0 1 + \ldots + c_{k} \left(\frac{\Omega_\ell}{2 \pi \varepsilon} \right)^{\!k}$};
	\node (inv-poly) [obj, below of=insert, align=center,yshift=-.25cm] {{\scriptsize Mixed form in $\Omega^{2*}\!(U_\ell)$:} \\ $\scriptstyle P\left(f\left(\frac{\Omega_\ell}{2 \pi \varepsilon}\right)\right)$};
	
	\begin{scope}[on background layer]
	\tikzset{loop/.style={draw,dashed,gray,rounded corners,fill=kwcolor!30,inner sep=10pt}}
	\tikzset{init/.style={draw,dashed,gray,rounded corners,fill=strcolor!30,inner sep=10pt}}
	\node[loop,fit=(insert) (inv-poly)] (l-box) {};
	\node[init,fit=(sym-expr) (taylor)] (i-box) {};
	\end{scope}
	
	\node (dict) [obj, right of=l-box, align=center,xshift=4cm] {{\scriptsize Dict. of curv. matr.:} \\ $\scriptstyle\left\{ e_1 \, : \, \Omega_1 , \, \ldots \, , \, e_m \, : \, \Omega_m \right\}$};
	
	\node (result) [obj, below of=inv-poly, align=center,yshift=-.5cm] {{\scriptsize Mixed form in $\Omega^{2*}\!(M)$:} \\ $\kappa(E, \nabla)$};
	
	\node [left of=i-box,darkgray,xshift=-1cm] {(1)};
	\node [left of=l-box,darkgray,xshift=-1.5cm] {(2)};
	
	\draw [arrow] (sym-expr) -- (sym-expr-mod) node[midway,right,darkgray] {\small \itshape transform};
	\draw [arrow] (sym-expr-mod) -- (taylor) node[midway,right,darkgray] {\small \itshape Taylor};
	\draw [arrow] (taylor) -- (insert) node[midway,right,darkgray] {\small \itshape functional calculus};
	\draw [arrow] (dict) -- (l-box) node[midway,above,darkgray] {\small\itshape for-loop} node[midway,below,darkgray] {$\scriptstyle \ell=1,\ldots, m$};
	\draw [arrow] (insert) -- (inv-poly) node[midway,right,darkgray] {\small \itshape apply $P$};
	\draw [arrow] (inv-poly) -- (result) node[midway,right,yshift=-5pt,darkgray] {\small \itshape set restriction};
	\end{tikzpicture}
	\caption{An illustration of the construction algorithm of a characteristic form in \Sage, starting from a symbolic expression and a dictionary of curvature matrices.}\label{fig:char:specific_workflow}
\end{figure}
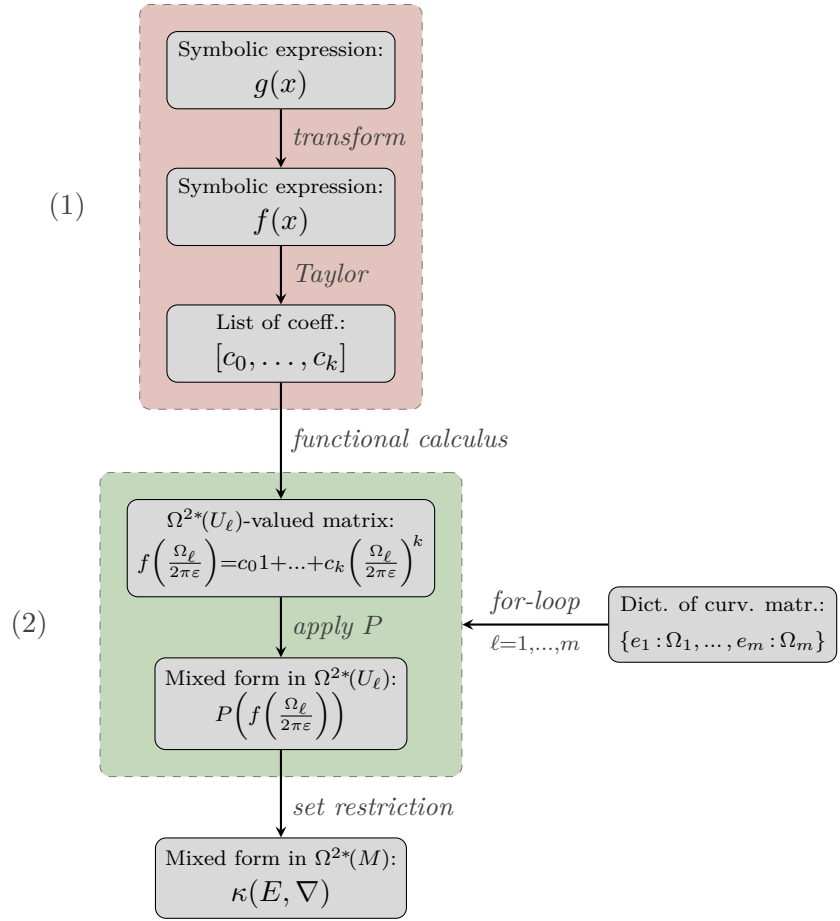

Figure~\ref{fig:char:specific_workflow}\,(1) still belongs to the initialization process of the characteristic class. Here, the given symbolic expression $g(x)$ undergoes a transformation into a new symbolic expression $f(x)$.\footnote{In fact, a new symbolic variable on the complex domain is established at this point. This precaution is due to the fact that symbolic variables are globally defined throughout an entire \Sage session.} The transformation depends on the underlying field and class type. We simply have $f(x)=g(x)$ if $E$ is complex. However, the real case is more complicated as discussed above; a list is provided in Table~\ref{tab:trafo_expr}. Afterwards, the Taylor expansion is computed at $x=0$. Fortunately, \Sage already comes with suitable Taylor expansion algorithms. Since every $2$-form-valued matrix is nilpotent of degree $\left\lceil \frac{n}{2} \right\rceil$, it is reasonable to consider expansions only up to $k = \left\lfloor \frac{n}{2} \right\rfloor$. Finally, the coefficients are stored in a \Python\ list $[c_0, \ldots, c_k]$ ready for use.

\begin{table}[t]
	\centering
	\def\arraystretch{1.4}
	\begin{tabular}{ll}
		\toprule
		\textbf{Class Type} & \textbf{Transformation}\\
		\midrule
		multiplicative & $f(x) = \sqrt{g(x^2)}$ \\
		additive & $f(x) = \frac{1}{2} g(x^2)$ \\
		Pfaffian & $f(x) = \frac{g(x)-g(-x)}{2}$ \\
		\bottomrule
	\end{tabular}
	\caption{All transformations $g(x) \to f(x)$ depending on the class type for vector bundles over the real field.}\label{tab:trafo_expr}
\end{table}

Figure~\ref{fig:char:specific_workflow}\,(2) illustrates the input/output machinery of the characteristic class invoked by the method \texttt{get\_form}. Once the \Python\ dictionary ${\{e_1 : \Omega_1, \ldots, e_m : \Omega_m \}}$ is given, a for-loop is established iterating over each pair $(e_\ell, \Omega_\ell)$ for $\ell=1, \ldots, m$. If the curvature form matrix $\Omega_\ell$ has not been converted into a proper algebraic \texttt{Sage} element yet, i.\,e.\ a generic matrix over the algebra $\Omega^*(U_\ell)$, then it is done here. We denote by $U_\ell$ the domain of the local frame~$e_\ell$. By using the coefficients $[c_0, \ldots, c_k]$ established during the initialization process, the functional calculus is applied: $$f\!\left(\frac{\Omega_\ell}{2 \pi \varepsilon}\right) = c_0 1 + c_1 \frac{\Omega_\ell}{2 \pi \varepsilon} + \ldots + c_k \left(\frac{\Omega_\ell}{2 \pi \varepsilon}\right)^{\!k}.$$ The placeholder $\varepsilon$ depends on the given class type. If $\kappa(E)$ is of Pfaffian type, we obtain $\varepsilon=1$. In any other case we get $\varepsilon=i$. Next the invariant polynomial $P$ is evaluated at the result. \Sage luckily supports determinants, traces and Pfaffians of matrices over arbitrary (commutative) rings. Hence, we end up with a closed form $P\left(f\!\left(\frac{\Omega_\ell}{2 \pi \varepsilon}\right)\right)$ living in $\Omega^{2*}(U_\ell) \subset \Omega^{*}(U_\ell)$. Finally representing the characteristic form, these iterated results are glued together in the mixed differential form $\kappa(E, \nabla) \in \Omega^{2*}(M) \subset \Omega^{*}(M)$ by using the method \texttt{set\_restriction} (cf. Section~\ref{sec:mixed:impl}).

\subsection{Example: Chern Character over Minkowski Space}\label{sec:char:ch_char}
We want to illustrate the usage of characteristic classes within \Sage by computing the Chern character form $\mathrm{ch}(E, \nabla^E)$ on a complex trivial line bundle $E$ over the 2-dimensional Minkowski space $M$ equipped with a bundle connection $\nabla^E$. We start with the general setup:
\begin{NBin}
M = Manifold(2, 'M', structure='Lorentzian')
X.<t,x> = M.chart()
E = M.vector_bundle(1, 'E', field='complex'); print(E)
\end{NBin}
\begin{NBout}
Differentiable complex vector bundle E -> M of rank 1 over the base space 2-dimensional Lorentzian manifold M
\end{NBout}
To trivialize the vector bundle $E$, we fix a global frame $e$:
\begin{NBin}
e = E.local_frame('e') # trivialize
\end{NBin}
Let us declare an $\mathrm{U}(1)$-connection $\nabla^E$ on $E$ given by an electromagnetic potential $A(t)$:
\begin{NBin}
nab = E.bundle_connection('nabla^E', latex_name=r'\nabla^E')
A = function('A')
\end{NBin}
The corresponding connection form $\omega$ turns out as:
\begin{NBin}
omega = M.one_form(name='omega', latex_name=r'\omega')
omega[1] = I*A(t)
omega.display()
\end{NBin}
\begin{NBoutM}
$\omega = i \, A\left(t\right) \mathrm{d} x$
\end{NBoutM}
Let us put this into the connection:
\begin{NBin}
nab.set_connection_form(0, 0, omega)
\end{NBin}
Notice that the Chern character $\mathrm{ch}(E)$ is already predefined in the system. We can get it by the following command:
\begin{NBin}
ch = E.characteristic_class('ChernChar'); print(ch)
\end{NBin}
\begin{NBout}
Characteristic class ch of additive type associated to e^x on the Differentiable complex vector bundle E -> M of rank 1 over the base space 2-dimensional Lorentzian manifold M
\end{NBout}
The computation of the corresponding Chern character form $\mathrm{ch}(E,\nabla^E)$ can be invoked by the method \texttt{get\_form}:
\begin{NBin}
ch_form = ch.get_form(nab)
ch_form.display_expansion()
\end{NBin}
\begin{NBoutM}
$\mathrm{ch}(E, \nabla^E) = \left[ 1 \right]_0 + \left[ 0 \right]_1 + \left[ \frac{\frac{\partial\,A}{\partial t}}{2 \, \pi} \mathrm{d} t\wedge \mathrm{d} x \right]_2$
\end{NBoutM}
We can see that the resulting 2-form coincides with the Faraday tensor divided by~$2 \pi$ as expected.

\subsection{Example: Chern Class of the Tautological Line Bundle}\label{sec:char:taut}
This essential example is primarily adopted from \cite[Beispiel~1.17]{baer:charakteristisch} and resembles the example we discuss in Section~\ref{sec:vec:moebius}. Before we start with the code, we define the \defstyle{tautological line bundle} over the complex projective space $\CP^n$:
\begin{align*}
	\gamma_n := \left\{ (v, \ell) \in \CC^{n+1} \times \CP^n ~|~ v \in \ell \cup \{0\} \right\},
\end{align*}
together with the footpoint map $\pi:\gamma_n \to \CP^n$ given by $(v, \ell) \mapsto \ell$. For now, we restrict to the case $n=1$. We choose homogeneous coordinates $$\CP^1=\left\{ [z_0:z_1] ~|~ (z_0,z_1) \in \CC^2 \setminus \{(0,0)\} \right\}$$ and define $U:= \CP^1 \setminus \left\{ [1:0] \right\}$ so that $\CC \to U$ given by $z \mapsto [z:1]$ defines a diffeomorphism. We start our computation by initializing the complex projective space as 2-dimensional real manifold with coordinates on $U$:
\begin{NBin}
M = Manifold(2, 'CP^1', start_index=1)
U = M.open_subset('U')
c_cart.<x,y> = U.chart() # [1:x+I*y]
\end{NBin}
For the sake of convenience, we additionally declare the complex coordinates $z$ and $\bar{z}$ on $U$:
\begin{NBin}
c_comp.<z, zbar> = U.chart(r'z:z zbar:\bar{z}')
cart_to_comp = c_cart.transition_map(c_comp, (x+I*y, x-I*y))
cart_to_comp.display()
\end{NBin}
\begin{NBoutM}
$\left\{\begin{array}{lcl} {z} & = & x + i \, y \\ {\bar{z}} & = & x - i \, y \end{array}\right.$
\end{NBoutM}
\begin{NBin}
comp_to_cart = cart_to_comp.inverse()
comp_to_cart.display()
\end{NBin}
\begin{NBoutM}
$\left\{\begin{array}{lcl} x & = & \frac{1}{2} \, {z} + \frac{1}{2} \, {\bar{z}} \\ y & = & -\frac{1}{2} i \, {z} + \frac{1}{2} i \, {\bar{z}} \end{array}\right.$
\end{NBoutM}
Now, we are ready to construct the tautological line bundle $\gamma_1$:
\begin{NBin}
E = M.vector_bundle(1, 'gamma_1', 
					latex_name=r'\gamma_1', 
		    		field='complex')
\end{NBin}
Furthermore, we declare a local frame $e$ on $U$ naturally given by ${[z:1] \mapsto \begin{smallpmatrix} z \\ 1 \end{smallpmatrix}}$:
\begin{NBin}
e = E.local_frame('e', domain=U)
\end{NBin}
To compute the Chern class, we still need a connection. The tautological line bundle inherits a Hermitian metric from the overlying trivial bundle $\CC^2 \times \CP^1$. Supposing $\pi_\ell: \CC^2 \to \ell \cup \{0\}$ denotes the orthogonal projection onto the complex line $\ell \in \CP^1$, the induced connection is given by
\begin{align*}
	\nabla_X \sigma = \left( \pi_\ell(\partial_X \nu ), \ell \right)
\end{align*}
for each $X \in T \CP^1$ and each section $\sigma: U \to \gamma_1$ written as $\ell \mapsto (\nu(\ell), \ell)$. Inserting our local frame $e$, we observe:
\begin{align*}
	\nabla_{\frac{\partial}{\partial \bar{z}}}\, e = 0, \qquad
	\nabla_{\frac{\partial}{\partial z}} e = \frac{\bar{z}}{1+ \abs{z}^2} \, e.
\end{align*}
Hence, we have obtained a suitable connection:
\begin{NBin}
nab = E.bundle_connection('nabla', latex_name=r'\nabla')
omega = U.one_form(name='omega')
omega[c_comp.frame(), 1, c_comp] = zbar/(1+z*zbar)
nab.set_connection_form(1, 1, omega, frame=e)
\end{NBin}
It is time to initialize $c(\gamma_1)$. Fortunately, \Sage already knows the notion Chern classes:
\begin{NBin}
c = E.characteristic_class('Chern'); print(c)
\end{NBin}
\begin{NBout}
Characteristic class c of multiplicative type associated to x + 1 on the Differentiable complex vector bundle gamma_1 -> CP^1 of rank 1 over the base space 2-dimensional differentiable manifold CP^1
\end{NBout}
Let us execute the algorithm:
\begin{NBin}
c_form = c.get_form(nab)
c_form.display_expansion(c_comp.frame(), chart=c_comp)
\end{NBin}
\begin{NBoutM}
$c(\gamma_1, \nabla) = \left[ 1 \right]_0 + \left[ 0 \right]_1 + \left[ \frac{i}{2 \, {\left(\pi + \pi {z}^{2} {\bar{z}}^{2} + 2 \, \pi {z} {\bar{z}}\right)}} \mathrm{d} {z}\wedge \mathrm{d} {\bar{z}} \right]_2$
\end{NBoutM}
Since this particular representation is defined outside a set of measure zero, we can compute its integral over $\CP^1$ in real coordinates:
\begin{NBin}
integrate(integrate(c_form[2][[1,2]].expr(c_cart), x, -infinity, infinity).full_simplify(), y, -infinity, infinity)
\end{NBin}
\begin{NBoutM}
$1$
\end{NBoutM}
The result shows that $c_1(\gamma_1)$ generates the second integer cohomology $H^2(\CP^1, \ZZ)$. Assume $\iota: \CP^1 \hookrightarrow \CP^n$ is the canonical embedding of projective lines into $\CC^{n+1}$. Then we observe $\iota^* \gamma_n = \gamma_1$, and due to the naturality condition $\iota^* c_1(\gamma_n) = c_1(\gamma_1) \neq 0$. In fact, by applying the Gysin sequence, one can conclude that $c_1(\gamma_n)$ even generates the cohomology ring $H^*(\CP^n, \ZZ)$ for each $n\in\NN$, see~\cite[Thm.~14.4]{milnor:char_classes}.

The preceding example plays an important role in the context of characteristic classes. Namely, if $E$ is a complex line bundle over some manifold $M$, there exists a continuous map $f : M \to \CP^n$ such that $E \cong f^* \gamma_n$, supposing $n$ is sufficiently large.\footnote{This particular statement is a corollary of \cite[Thm.~7.1, p.~77]{walschap:metric_structures} where the proof can be easily adapted to the complex field. Notice that there is a more general result~\cite[Thm.~5.6]{milnor:char_classes} regarding vector bundles over arbitrary paracompact spaces.} Once such a map is known, all characteristic classes of $E$ are completely determined by~$\gamma_n$.

\subsection{Example: Euler Class of $\mathbb{S}^2$}\label{sec:char:euler}
In this example, we want to compute the Euler class of the 2-sphere $\mathbb{S}^2 \subset \RR^3$. As usual, we cover $\Sphere^2$ by two parallelizable open subsets $U:=\Sphere^2 \setminus \{(0,0,1)\}$ and $V:=\Sphere^2 \setminus \{(0,0,-1)\}$ for which the point $(0,0,1)$ is identified with the north pole. We define stereographic coordinates in \Sage: 
\begin{NBin}
M = Manifold(2, name='S^2', latex_name=r'\mathbb{S}^2', 
			 structure='Riemannian', start_index=1)
U = M.open_subset('U') ; V = M.open_subset('V')
M.declare_union(U,V)   # M is the union of U and V
stereoN.<x,y> = U.chart()
stereoS.<xp,yp> = V.chart("xp:x' yp:y'")
N_to_S = stereoN.transition_map(stereoS,
								(x/(x^2+y^2), y/(x^2+y^2)),
								intersection_name='W',
								restrictions1= x^2+y^2!=0,
								restrictions2= xp^2+yp^2!=0)
S_to_N = N_to_S.inverse()
\end{NBin}
Next we define the tangent bundle and its local frames induced by the charts:
\begin{NBin}
eU = stereoN.frame(); eV = stereoS.frame()
TM = M.tangent_bundle()
\end{NBin} 
The Euler class is also one of the predefined classes. Thus, we can easily get it from the tangent bundle's instance:
\begin{NBin}
e_class = TM.characteristic_class('Euler'); print(e_class)
\end{NBin}
\begin{NBout}
Characteristic class e of Pfaffian type associated to x on the Tangent bundle TS^2 over the 2-dimensional Riemannian manifold S^2
\end{NBout}
To compute a form representing the Euler class, we need to state a suitable connection first. Here, we want to use the Levi--Civita connection induced by the standard metric. This is simply given by the pullback of the Euclidean scalar product of the ambient space $\RR^3$ along the canonical embedding $\iota: \Sphere^2 \hookrightarrow \RR^3$. Let us define the ambient space $\RR^3$ and its Euclidean scalar product $h$:
\begin{NBin}
E = Manifold(3, 'R^3', latex_name=r'\mathbb{R}^3', start_index=1)
cart.<X,Y,Z> = E.chart()
h = E.metric('h')
h[1,1], h[2,2], h[3, 3] = 1, 1, 1
h.display()
\end{NBin}
\begin{NBoutM}
$h = \mathrm{d} X\otimes \mathrm{d} X+\mathrm{d} Y\otimes \mathrm{d} Y+\mathrm{d} Z\otimes \mathrm{d} Z$
\end{NBoutM}
On that account, we declare the embedding $\iota: \Sphere^2 \hookrightarrow \RR^3$ in stereographic coordinates when one considers its projection from the north pole $(0, 0, 1)$ to the equatorial plane~${Z=0}$:
\begin{NBin}
iota = M.diff_map(E, {(stereoN, cart): 
					  [2*x/(1+x^2+y^2), 2*y/(1+x^2+y^2),
					  (1-x^2-y^2)/(1+x^2+y^2)],
					  (stereoS, cart): 
					  [2*xp/(1+xp^2+yp^2), 2*yp/(1+xp^2+yp^2),
					  (xp^2+yp^2-1)/(1+xp^2+yp^2)]},
				  name='iota', latex_name=r'\iota')
iota.display()
\end{NBin}
\begin{NBoutM}
$\begin{array}{llcl} \iota:& \mathbb{S}^2 & \longrightarrow & \mathbb{R}^3 \\ \mbox{on}\ U : & \left(x, y\right) & \longmapsto & \left(X, Y, Z\right) = \left(\frac{2 \, x}{x^{2} + y^{2} + 1}, \frac{2 \, y}{x^{2} + y^{2} + 1}, \frac{x^{2} + y^{2} - 1}{x^{2} + y^{2} + 1}\right) \\ \mbox{on}\ V : & \left({x'}, {y'}\right) & \longmapsto & \left(X, Y, Z\right) = \left(\frac{2 \, {x'}}{{x'}^{2} + {y'}^{2} + 1}, \frac{2 \, {y'}}{{x'}^{2} + {y'}^{2} + 1}, -\frac{{x'}^{2} + {y'}^{2} - 1}{{x'}^{2} + {y'}^{2} + 1}\right) \end{array}$
\end{NBoutM}
We can define the standard metric $g$ on $\Sphere^2$ by setting it as the pullback metric~$\iota^*h$:
\begin{NBin}
g = M.metric()
g.set(iota.pullback(h))
g[1,1].factor(); g[2,2].factor() # simplifications
g.display()
\end{NBin}
\begin{NBoutM}
$g = \frac{4}{{\left(x^{2} + y^{2} + 1\right)}^{2}} \mathrm{d} x\otimes \mathrm{d} x + \frac{4}{{\left(x^{2} + y^{2} + 1\right)}^{2}} \mathrm{d} y\otimes \mathrm{d} y$
\end{NBoutM}
The corresponding Levi--Civita connection is computed automatically:
\begin{NBin}
nab = g.connection()
\end{NBin}
Since we have found the desired Levi--Civita connection, we want to compute the associated curvature forms and store them in a \Python\ list:
\begin{NBin}
cmatrix_U = [[nab.curvature_form(i,j,eU) for j in TM.irange()]
              for i in TM.irange()]
cmatrix_V = [[nab.curvature_form(i,j,eV) for j in TM.irange()]
              for i in TM.irange()]
\end{NBin}
Fortunately, the curvature form matrices are already skew-symmetric:
\begin{NBin}
for i in range(TM.rank()):
    for j in range(TM.rank()):
        show(cmatrix_U[i][j].display())
\end{NBin}
\begin{NBoutM}
$\Omega^1_{\ \, 1} = 0 \\
\Omega^1_{\ \, 2} = \left( \frac{4}{x^{4} + y^{4} + 2 \, {\left(x^{2} + 1\right)} y^{2} + 2 \, x^{2} + 1} \right) \mathrm{d} x\wedge \mathrm{d} y \\
\Omega^2_{\ \, 1} = \left( -\frac{4}{x^{4} + y^{4} + 2 \, {\left(x^{2} + 1\right)} y^{2} + 2 \, x^{2} + 1} \right) \mathrm{d} x\wedge \mathrm{d} y \\
\Omega^2_{\ \, 2} = 0$
\end{NBoutM}
\begin{NBin}
for i in range(TM.rank()):
	for j in range(TM.rank()):
		show(cmatrix_V[i][j].display())
\end{NBin}
\begin{NBoutM}
$\Omega^1_{\ \, 1} = 0 \\
\Omega^1_{\ \, 2} = \left( \frac{4}{{x'}^{4} + {y'}^{4} + 2 \, {\left({x'}^{2} + 1\right)} {y'}^{2} + 2 \, {x'}^{2} + 1} \right) \mathrm{d} {x'}\wedge \mathrm{d} {y'} \\
\Omega^2_{\ \, 1} = \left( -\frac{4}{{x'}^{4} + {y'}^{4} + 2 \, {\left({x'}^{2} + 1\right)} {y'}^{2} + 2 \, {x'}^{2} + 1} \right) \mathrm{d} {x'}\wedge \mathrm{d} {y'} \\
\Omega^2_{\ \, 2} = 0$
\end{NBoutM}
Hence, we can put them into a dictionary and apply the algorithm:
\begin{NBin}
cmatrices = {eU: cmatrix_U, eV: cmatrix_V}
e_class_form = e_class.get_form(nab, cmatrices)
e_class_form.display_expansion()
\end{NBin}
\begin{NBoutM}
$e(T\mathbb{S}^2, \nabla_g) = \left[ 0 \right]_0 + \left[ 0 \right]_1 + \left[ \left( \frac{2}{\pi + \pi x^{4} + \pi y^{4} + 2 \, \pi x^{2} + 2 \, {\left(\pi + \pi x^{2}\right)} y^{2}} \right) \mathrm{d} x\wedge \mathrm{d} y \right]_2$
\end{NBoutM}
We want to compute the Euler characteristic of $\Sphere^2$ now. Due to Theorem~\ref{thm:gauss_bonnet_chern}, this can be achieved by integrating the top form over $\Sphere^2$. Since $U$ and $\Sphere^2$ differ only by a point, and therefore a set of measure zero, it is sufficient to integrate over the subset $U$:
\begin{NBin}
integrate(integrate(e_class_form[2][[1,2]].expr(), x, -infinity, infinity).simplify_full(), y, -infinity, infinity)
\end{NBin}
\begin{NBoutM}
$2$
\end{NBoutM}
Thus, we have obtained the Euler characteristic of $\mathbb{S}^2$.

\subsection{Example: $\hat{A}$-Class of Lorentzian Foliation of Berger Spheres}\label{sec:char:berger}
We consider the space of unit quaternions $\Sphere^3 \subset \RR^4 \cong \HH$, where $\HH$ is endowed with the canonical basis $(\mathbf{1}, \mathbf{i}, \mathbf{j}, \mathbf{k})$. Note that the product $qp$ is tangent to $\Sphere^3$ for each $q \in \Sphere^3$ whenever $p$ is imaginary. This gives rise to the following vector fields on~$\Sphere^3$:
\begin{align*}
	\restrict{\varepsilon_1}{q} = q \,\mathbf{i}, \quad \restrict{\varepsilon_2}{q} = q \,\mathbf{j}, \quad \restrict{\varepsilon_3}{q} = q \,\mathbf{k}.
\end{align*}
It can be shown that $(\varepsilon_1, \varepsilon_2, \varepsilon_3)$ is a global vector frame and hence $\Sphere^3$ is a parallelizable manifold.

We introduce the following smooth family of the so-called \defstyle{Berger metrics}:
\begin{align*}	
	g_t = a(t)^{2} \; \varepsilon^{1}\otimes \varepsilon^{1} +\varepsilon^{2}\otimes \varepsilon^{2} +\varepsilon^{3}\otimes \varepsilon^{3}.
\end{align*}
Here, $a(t)$ is a smooth positive function in $t \in \RR$ and $\varepsilon^i$ denotes the vector field dual to $\varepsilon_i$. This family can be used to define a globally hyperbolic manifold $M=\RR \times \Sphere^3$ equipped with the Lorentzian metric $g = - {\diffd t}^2 + g_t$. Hence, $M$ is foliated by Berger spheres.

In the following, we compute the $\hat{A}$-form of the corresponding Levi--Civita connection $\nabla_g$. Even though the $\hat{A}$-class vanishes, its characteristic form still plays an significant role in the index theory of the classical Dirac operator when one considers spacelike Cauchy boundary~\cite{baer:math_chiral}. We start the computation by declaring the Lorentzian manifold first:
\begin{NBin}
M = Manifold(4, 'M', structure='Lorentzian'); print(M)
\end{NBin}
\begin{NBout}
4-dimensional Lorentzian manifold M
\end{NBout}
We cover $M$ by two open subsets defined as $U := \RR \times \left( \Sphere^3 \setminus \{ -\mathbf{1} \} \right)$ and $V := \RR \times \left( \Sphere^3 \setminus \{ \mathbf{1} \} \right)$:
\begin{NBin}
U = M.open_subset('U'); V = M.open_subset('V')
M.declare_union(U,V)
\end{NBin}
We need to impose coordinates on $M$ and use stereographic projections with respect to the foliated 3-sphere:
\begin{NBin}
stereoN.<t,x,y,z> = U.chart()
stereoS.<tp,xp,yp,zp> = V.chart("tp:t' xp:x' yp:y' zp:z'")
N_to_S = stereoN.transition_map(stereoS,
								(t, x/(x^2+y^2+z^2), 
									y/(x^2+y^2+z^2), 
									z/(x^2+y^2+z^2)),
								intersection_name='W',
								restrictions1= x^2+y^2+z^2!=0,
								restrictions2= xp^2+yp^2+zp^2!=0)
W = U.intersection(V)
S_to_N = N_to_S.inverse()
N_to_S.display()
\end{NBin}
\begin{NBoutM}
$\left\{\begin{array}{lcl} {t'} & = & t \\ {x'} & = & \frac{x}{x^{2} + y^{2} + z^{2}} \\ {y'} & = & \frac{y}{x^{2} + y^{2} + z^{2}} \\ {z'} & = & \frac{z}{x^{2} + y^{2} + z^{2}} \end{array}\right.$
\end{NBoutM}
From above, we know that $M$ admits a global frame $(\varepsilon_0, \varepsilon_1, \varepsilon_2, \varepsilon_3)$:
\begin{NBin}
E = M.vector_frame('E', latex_symbol=r'\varepsilon')
E_U = E.restrict(U); E_U
\end{NBin}
\begin{NBoutM}
$\left(U, \left(\varepsilon_{0},\varepsilon_{1},\varepsilon_{2},\varepsilon_{3}\right)\right)$
\end{NBoutM}
The vector field $\varepsilon_0$ is simply given by $\frac{\partial}{\partial t}$. To obtain the global vector frame $(\varepsilon_1, \varepsilon_2, \varepsilon_3)$ on $\Sphere^3$ in stereographic coordinates, a computation within \Sage is performed in the \texttt{Jupyter} notebook \enquote{Sphere S3: vector fields and left-invariant parallelization} downloadable from~\cite{gour:examples}. This is done by embedding $\Sphere^3$ into $\RR^4 \cong \HH$, endowing it with the quaternionic structure. This eventually leads to:
\begin{NBin}
E_U[0][:] = [1,0,0,0]
E_U[1][:] = [0, (x^2-y^2-z^2+1)/2, x*y+z, x*z-y]
E_U[2][:] = [0, x*y-z, (1-x^2+y^2-z^2)/2, x+y*z]  
E_U[3][:] = [0, x*z+y, y*z-x, (1-x^2-y^2+z^2)/2]
\end{NBin}
\begin{NBin}
for i in M.irange():
	show(E_U[i].display())
\end{NBin}
\begin{NBoutM}
$\varepsilon_{0} = \frac{\partial}{\partial t } \\
\varepsilon_{1} = \left( \frac{1}{2} \, x^{2} - \frac{1}{2} \, y^{2} - \frac{1}{2} \, z^{2} + \frac{1}{2} \right) \frac{\partial}{\partial x } + \left( x y + z \right) \frac{\partial}{\partial y } + \left( x z - y \right) \frac{\partial}{\partial z } \\
\varepsilon_{2} = \left( x y - z \right) \frac{\partial}{\partial x } + \left( -\frac{1}{2} \, x^{2} + \frac{1}{2} \, y^{2} - \frac{1}{2} \, z^{2} + \frac{1}{2} \right) \frac{\partial}{\partial y } + \left( y z + x \right) \frac{\partial}{\partial z } \\
\varepsilon_{3} = \left( x z + y \right) \frac{\partial}{\partial x } + \left( y z - x \right) \frac{\partial}{\partial y } + \left( -\frac{1}{2} \, x^{2} - \frac{1}{2} \, y^{2} + \frac{1}{2} \, z^{2} + \frac{1}{2} \right) \frac{\partial}{\partial z }$
\end{NBoutM}
\noindent To ensure evaluations in this particular frame, we must communicate the change-of-frame formula to \Sage:
\begin{NBin}
P = U.automorphism_field()
for i in M.irange():
	for j in M.irange():
		P[j,i] = E_U[i][j]
\end{NBin}
\begin{NBin}
U.set_change_of_frame(stereoN.frame(), E_U, P)
\end{NBin}
The subset $U$ differs from $M$ only by a one dimensional slit and is therefore dense in $M$. As we know that $\left(\varepsilon_{0},\varepsilon_{1},\varepsilon_{2},\varepsilon_{3}\right)$ defines a \emph{global} frame, its components can be easily and uniquely extended to all of $M$. For this, we use the method \texttt{add\_comp\_by\_continuation}:
\begin{NBin}
for i in M.irange():
	E[i].add_comp_by_continuation(stereoS.frame(), W)
\end{NBin}
And again, we declare the change of frame:
\begin{NBin}
P = V.automorphism_field()
for i in M.irange():
	for j in M.irange():
		P[j,i] = E.restrict(V)[i][j]
\end{NBin}
\begin{NBin}
V.set_change_of_frame(stereoS.frame(), E.restrict(V), P)
\end{NBin}
In order to reduce the computation time, we examine the $\hat{A}$-class on the open subset $U \subset M$ first. The final result can be obtained by continuation. For this purpose, we define the tangent bundle over $U$:
\begin{NBin}
TU = U.tangent_bundle(); print(TU)
\end{NBin}
\begin{NBout}
Tangent bundle TU over the Open subset U of the 4-dimensional Lorentzian manifold M
\end{NBout}
Notice that the $\hat{A}$-class is already predefined:
\begin{NBin}
A = TU.characteristic_class('AHat'); A
\end{NBin}
\begin{NBoutM}
$\hat{A}(TU)$
\end{NBoutM}
As discussed before, its holomorphic function is given by:
\begin{NBin}
A.function()
\end{NBin}
\begin{NBoutM}
$\frac{\sqrt{x}}{2 \, \sinh\left(\frac{1}{2} \, \sqrt{x}\right)}$
\end{NBoutM}
We are ready to define the Berger metric, at least on the subset $U$:
\begin{NBin}
a = function('a')
\end{NBin}
\begin{NBin}
g = U.metric()
g.add_comp(E_U)[0, 0] = - 1
g.add_comp(E_U)[1, 1] = a(t)^2
g.add_comp(E_U)[2, 2] = 1
g.add_comp(E_U)[3, 3] = 1
g.display(E_U)
\end{NBin}
\begin{NBoutM}
$g = -\varepsilon^{0}\otimes \varepsilon^{0} + a\left(t\right)^{2} \varepsilon^{1}\otimes \varepsilon^{1} +\varepsilon^{2}\otimes \varepsilon^{2} +\varepsilon^{3}\otimes \varepsilon^{3}$
\end{NBoutM}
The corresponding connection is automatically computed by \Sage:
\begin{NBin}
nab = g.connection(); nab
\end{NBin}
\begin{NBoutM}
$\nabla_g$
\end{NBoutM}
Finally, we perform the computation of the $\hat{A}$-form with respect to this connection~$\nabla_g$:
\begin{NBin}
A_form = A.get_form(nab) # long time
A_form.display_expansion(E_U, stereoN)
\end{NBin}
\begin{NBoutM}
$\hat{A}(TU, \nabla_g) = \left[ 1 \right]_0 + \left[ 0 \right]_1 + \left[ 0 \right]_2 + \left[ 0 \right]_3 \\ \phantom{\hat{A}(TU, \nabla_g) = \left[ 1 \right]_0 } + \left[ \left( \frac{4 \, {\left(a\left(t\right)^{3} - a\left(t\right)\right)} \frac{\partial\,a}{\partial t} - \frac{\partial\,a}{\partial t} \frac{\partial^2\,a}{\partial t ^ 2}}{24 \, \pi^{2}} \right) \varepsilon^{0}\wedge \varepsilon^{1}\wedge \varepsilon^{2}\wedge \varepsilon^{3} \right]_4$
\end{NBoutM}
To attain $\hat{A}(TM, \nabla_g)$ in all given coordinates, we still have to extend the result onto~$M$. With respect to the global frame $(\varepsilon_0, \varepsilon_1,\varepsilon_2,\varepsilon_3)$, the form $\hat{A}(TU, \nabla_g)$ only depends on the global coordinate $t$. This makes the continuation trivial. Besides, one is interested, for the most part, in characteristic forms outside a set of measure zero. Hence, we terminate our calculation at this point.

\section{Conclusion and Perspectives}\label{sec:perspective}
\subsection{Summary and Conclusion}
All parts presented in this thesis are fully implemented and integrated into the \Sage project and are officially available since version 9. With over $14,\!000$ lines of new code, this includes: general improvements of existing code, vector bundles, sections, bundle connections, mixed differential forms and finally characteristic classes. All these extensions yield new tools for both teaching and research. In education, they can be used to create new examples to illustrate the notion of vector bundles and characteristic classes. With respect to research, recall that on manifolds without boundary, two characteristic forms of the same characteristic class have \emph{topologically invariant integrals}. That is, integrating each top form over the whole manifold yields the same result; which follows immediately from Stokes's theorem. Since all computations are based on explicit expressions, the corresponding integrals can be at least evaluated numerically. Therefore, our current implementation provides a functional tool for applications in differential topology as well as index theory and allows for the comparison of differential geometric and topological invariants on manifolds with an empty or, at least, specifically chosen boundary. A possible extension to manifolds with arbitrary, non-empty boundary is discussed in the following section.

\begin{table}[t]
	\centering
	\begin{tabular}{lr}
		\toprule
		\textbf{Example} & \textbf{Wall Time} \\
		\midrule
		\hyperref[in:78]{\color{incolor}\texttt{In [78]}} & \texttt{299\,ms} \\
		\hyperref[in:86]{\color{incolor}\texttt{In [86]}} & \texttt{560\,ms} \\
		\hyperref[in:98]{\color{incolor}\texttt{In [98]}} & \texttt{2.2\,s} \\
		\hyperref[in:117]{\color{incolor}\texttt{In [117]}} & \texttt{1\,h\ 2\,min\ 55\,s} \\
		\bottomrule
	\end{tabular}
	\caption{Preformed on an \texttt{Intel(R) Core(TM) i3-3120M CPU @ 2.50GHz} machine using \texttt{Ubuntu 18.04.3 LTS} and \Sage version \texttt{9.0.beta1} without parallelization.}\label{tab:con:times}
\end{table}

\paragraph{Wall Times.}
In Table~\ref{tab:con:times}, we list the wall times of our computations of characteristic forms performed in Section~\ref{sec:char:ch_char},~\ref{sec:char:taut},~\ref{sec:char:euler} and~\ref{sec:char:berger}. The first three examples yield short computation times as expected. The dimension of the base space and the rank of the vector bundle do not exceed two, which is computationally simple for \Sage. In contrast, the last example yields a wall time of more than one hour. Here, we are dealing with four dimensions and more complicated symbolic expressions. On top of that, parallelization is deactivated and division free algorithms are used in the background. These algorithms are not optimized for speed. Notice that $\Sphere^3$ is a Lie group, and using structure coefficients would provide a shortcut in the calculation. However, since general Lie groups are currently not supported, stating the global frame in stereographic coordinates is the only viable option.

\subsection{Outlook and Future Prospects}
As an open-source piece of software, \Sage is always in a state of progress; and there are tickets still open at this point.\footnote{As of 02/2020, the tickets \ticket{28629}, \ticket{28640}, \ticket{28854} and \ticket{28963} are still open.} Nevertheless, our modifications provide a solid starting point and already provide a powerful tool. In this section, we give a glimpse of possible extensions to our work and collect some ideas for future versions of \Sage.

\paragraph{Vector Bundles.} Regarding vector bundles, there is still much work to do. Many features are yet unimplemented. This specifically involves bundle metrics and bundle automorphisms as well as tensor product bundles, duals and pullbacks.\footnote{The latter are not a severe restriction as they can still be considered as separate vector bundles and initialized manually.} To achieve this, the current tensor field code can be generalized to vector bundles again. Undertaking this task, it is equally beneficial to merge tensor fields entirely into the setup of vector bundles. This particularly includes structured inheritance trees in order to remove code~duplication. %

Another issue concerns bundle connections. We are naturally interested in applying them to local sections. This could be accomplished as follows. Suppose we have a local section and a vector field defined on the same domain. By using the decomposition in~\eqref{eq:vec:sec_decomp} and making use of~\eqref{eq:vec:def_con_form} together with Leibniz's rule, we can derive a straightforward formula in terms of connection forms. This certainly involves actions of one forms on vector fields. Fortunately, they are integrated in \Sage already. Since each bundle connection is designed to store its connection forms, this formula could easily be implemented.

\paragraph{Characteristic Classes.} With respect to the Chern--Weil theory, our implementation already provides a comprehensive tool for applications. Nevertheless, there are still some things left to do. Currently, skew-symmetric curvature matrices must be inserted manually to obtain the characteristic form of a Pfaffian class. This, however, is not necessary once an orientation is chosen. In fact, one might find a general formula which could be implemented directly. Such a formula exists at least by applying the Gram--Schmidt procedure to local frames. Still, there might be a more concise formula that is easier to apply.

As Table~\ref{tab:con:times} indicates, computation times rise quickly with increasing complexity. For that purpose, it might be useful to support parallelization with respect to characteristic forms. One might also think about more sophisticated matrix algorithms for background tasks, especially in view of mixed differential forms. This could decrease computation times and make our implementation even more useful for applications in general relativity and mathematics.

As aforementioned, if our manifold has non-empty boundary, two characteristic forms in general no longer yield topologically invariant integrals. More precisely, both forms differ by an exact differential form which is the exterior derivative of the so-called \defstyle{transgression form}. Applying Stokes's theorem, its integral over the boundary yields the difference of these integrals. Transgression forms play an important role in gravitational physics as they occur, for example, in the index formula for the Dirac operator on Lorentzian manifolds~\cite{baer:math_chiral}. A concrete formula that could be embedded into \Sage is presented in~\cite[p.~299, Eq.~(6.4)]{eguchi:grav}. The way we designed characteristic classes within \Sage provides a solid foundation for this purpose.

Unfortunately, Stiefel--Whitney classes are still missing in \Sage. They are characteristic classes over the trivial field $\ZZ / 2 \ZZ$ and encode essential information about the underlying vector bundle. For example, the first Stiefel--Whitney class describes whether a vector bundle is orientable or not~\cite[Thm.~1.2]{lawson:spin}. The second Stiefel--Whitney class, on the other hand, provides us with information about the existence of a spin structure on an orientable manifold~\cite[Thm.~1.7]{lawson:spin}. However, as the de~Rham cohomology has coefficients in $\RR$ or $\CC$, there is little hope to obtain these classes by adopting our approach in Chern--Weil theory.
\paragraph{Lie Groups.}
Having the additional structure of a Lie group, most computations might get simpler. In particular, the Levi--Civita connection can be easily computed when using structure coefficients. Up to now, \Sage supports only \emph{nilpotent} Lie groups~\cite{sagemath:reference:nilpotent}. The general case would be a tremendous advantage to facilitate computations and enhance usability. Possible benefits can be inferred directly from Section~\ref{sec:char:berger} where the computation has seen to be expensive.

\phantomsection	
\addcontentsline{toc}{section}{References}
\begin{spacing}{1}
	\AtNextBibliography{\footnotesize}
	\printbibliography
\end{spacing}

\end{document}